\documentclass[11pt, leqno]{amsart}

\usepackage{color}
\usepackage{amsmath,amssymb,amsfonts,amsthm,bm}
\usepackage{graphicx}
\usepackage{float}
\usepackage{enumerate}
\usepackage{appendix}
\usepackage[numbers, sort&compress]{natbib}
 \usepackage[left=1.35in, right=1.35in]{geometry}
\def\R{\mathbb R}
\def\Z{\mathbb Z}
\def\T{\mathbb T}
\def\p{\mathbb P}
\def\q{\mathbb Q}

\def\Q{\mathbb Q}

\def\cC{\mathcal C}
\def\d{\partial}

\def\t{\dot}

\def\H{\dot{H}}

\def\rG{\tilde{G}}
\def\rP{\tilde{P}}
\def\aG{\bar{G}}
\def\aP{\bar{P}}

\renewcommand\div{{\rm div}\,}
\renewcommand\lim{{\rm lim}\,}
\renewcommand\exp{{\rm exp}\,}
\renewcommand\sup{{\rm sup}\,}
\renewcommand\inf{{\rm inf}\,}
\renewcommand\log{{\rm log}\,}

\renewcommand\max{{\rm max}\,}

\newcommand{\with}{\quad\!\hbox{with}\!\quad}
\newcommand{\andf}{\quad\!\hbox{and}\!\quad}
\newcommand{\ifa}{\quad\!\hbox{if}\!\quad}

\newcommand{\Int}{\displaystyle \int}

\newcommand{\dbar}{\,d\hspace*{-0.08em}\bar{}\hspace*{0.1em}}

\newtheorem{theorem}{Theorem}[section]
 \newtheorem{corollary}[theorem]{Corollary}
 \newtheorem{lemma}[theorem]{Lemma}
 \newtheorem{proposition}[theorem]{Proposition}
 \theoremstyle{definition}
 
 \theoremstyle{remark}
 \newtheorem{remark}[theorem]{Remark}
 
 \numberwithin{equation}{section}

\newcommand{\wt}{\widetilde}
\newcommand{\eps}{\varepsilon}
\newcommand{\divv}{\mbox{div }}

\newcommand{\inn}{\mbox{in}}
\newcommand{\abs}[1]{\left\vert#1\right\vert}

\newcommand{\norm}[1]{\Vert#1\Vert}

\begin{document}

\title[Exponential decay of compressible Navier-Stokes equations]{Exponential decay  results for compressible viscous flows with vacuum}
\author{ Rapha\"el Danchin}
\author{Shan Wang}

\begin{abstract}
We are concerned with
the isentropic compressible Navier-Stokes system   in the  two-dimensional torus, with 
rough data and vacuum  : the initial velocity  is in the Sobolev space $H^1$ and   the initial density is only bounded and nonnegative. Arbitrary regions of vacuum are admissible, and no compatibility condition is required. 
Under these assumptions and for large enough bulk viscosity,  global solutions have been constructed in  \cite{DM-ripped}. 
The main goal of the paper is to establish  that these solutions converge  exponentially fast to a constant state, and  to 
specify the convergence rate in terms of the   viscosity coefficients.

We also  prove  exponential decay estimates for the solutions
to the inhomogeneous incompressible Navier-Stokes equations. 
This latter result extends to the torus case the recent paper \cite{DMP} dedicated to this system 
in smooth bounded domains. 
\end{abstract}
\date{}

\maketitle
\section*{Introduction}
In this article, our main object of study is the  long-time behavior of global solutions to the 
barotropic compressible Navier-Stokes system in  a general two-dimensional periodic box 
$\T^2$. This system that governs the evolution of the density $\rho=\rho(t,x)\in[0,\infty)$ and of the velocity field
$u=u(t,x)\in\R^2$  of compressible viscous flows  reads:
\begin{equation*}
\left\{\begin{aligned}
&\rho_{t}+\div(\rho u)=0    & \inn \quad \R_+\times \T^{2}, \\
 &(\rho u)_{t}+\div(\rho u\otimes u)-\mu\Delta u-(\lambda\!+\!\mu)\nabla \div u+\nabla P=0 & \inn \quad\R_+\times \T^{2}, \\
 &(\rho,u)|_{t=0}=(\rho_{0},u_{0}) & \inn \quad \T^{2}.
\end{aligned}\right.\leqno(CNS)
\end{equation*}
Above, the pressure $P$ is a \emph{given} function of the density and the real numbers $\lambda$ and $\mu$ are the bulk and shear viscosity coefficients, respectively, they are assumed to satisfy
\begin{equation}\label{eq:viscosity}
\mu>0 \andf \nu:=\lambda+2\mu>0.\end{equation}
It is well known that sufficiently smooth solutions of  (CNS) satisfy for all $t\geq 0:$
\begin{itemize}
    \item the conservation of total mass and momentum, namely
    \begin{equation}\label{conlaws}
        \int_{\T^2}\rho(t,x)\,dx=\int_{\T^2}\rho_0(x)\,dx \andf \int_{\T^2}(\rho u)(t,x)\,dx=\int_{\T^2}(\rho_0 u_0)(x)\,dx;
    \end{equation}
    \item the following energy balance:
    \begin{multline}\label{eq:energy}
        E(t)+\int_0^t(\mu \|\nabla \p u(\tau)\|^2_{L_{2}}+\nu \|\div u(\tau)\|^2_{L_{2}})\,d\tau=E(0)=:E_0 \\
    \with E(t):=\int_{\T^2}\Bigl(\frac 12\rho(t,x)\abs{u(t,x)}^2+e(t,x)\Bigr)\,d\tau.
     \end{multline}
     \end{itemize}
     Above,  $e$ stands for the potential energy of  the fluid  defined   (up to a linear function) by 
\begin{equation}\label{eq:e}
   e(\rho):=\rho \int_{\bar\rho}^{\rho} \frac{P(s)}{s^2}\,ds-\frac{P(\bar\rho)}{\bar\rho}(\rho-\bar\rho)
   =\rho\int_{\bar\rho}^\rho\frac{P(s)-P(\bar\rho)}{s^2}\,ds,\qquad\bar\rho\in\R_+. 
\end{equation}
For simplicity, we here focus on the \emph{isentropic} case, namely 
\begin{equation}\label{eq:isentro}
P(\rho)=\kappa \rho^\gamma\with \kappa>0\andf \gamma\geq1.\end{equation}
Note that for $\kappa=\bar\rho=1,$ we have
\begin{equation}
e(\rho)=\begin{cases}
\rho \log \rho-\rho+1 & \ifa \gamma=1,\\
\frac{\rho^\gamma}{\gamma-1}-\frac{\gamma \rho}{\gamma-1}+1 & \ifa \gamma>1.
\end{cases}
\end{equation}
 Since the 60ies, a number of mathematical works have been dedicated to studying 
 the well-posedness of (CNS)  in various domains. On the one hand, whenever the data are smooth enough and bounded away from vacuum (i.e. $\inf\rho_0>0$), 
one can construct smooth local-in-time solutions \cite{Nash}, that  are global in the perturbative regime
of a stable constant solution \cite{MN}. 
On the other hand, by taking advantage of the energy balance \eqref{eq:energy}, one can produce global-in-time
weak solutions with finite energy \cite{Lions,Feireisl}.
\smallbreak
Here we are interested in an intermediate situation, first considered by D. Hoff in \cite{Hoff}, where the density has no regularity
but where the velocity  is more regular than required for \eqref{eq:energy}. 
We shall adopt the framework that has been introduced by B. Desjardins in \cite{Des-CPDE}
and used more recently by the first author in a recent joint paper with P.B. Mucha \cite{DM-ripped}: 
 the  initial velocity $u_0$ is in $H^1(\T^2)$ and the initial density $\rho_0$ is  only bounded: 
\begin{equation}\label{eq:inir}0\leq \rho_0\leq \rho^*_0:=\underset{x\in\T^2}{\sup}\rho_0(x)<\infty.\end{equation}
In the smooth situation with  periodic boundary conditions, it is expected  that  global solutions 
with `nice' properties converge exponentially fast to $(\bar\rho,0)$ with $\bar\rho$ being the mean value of $\rho_0,$ 
when time goes to infinity.  This fact has been proved recently for the full Navier-Stokes equations in  $\T^3,$ see \cite{ZZ}, and in smooth bounded domains \cite{DT}. 
It  turns out to be also true for not so smooth data  with, possibly, vacuum. 
In this direction, one may mention the  recent paper \cite{WYZ} by G. Wu, L. Yao and Y. Zhang. 
 There, density is in $W^{1,q}(\T^3)$ for some $q>3$
(which, unfortunately  precludes handling discontinuous densities)   but not necessarily positive everywhere, and the velocity is in $H^2(\T^3).$ 
\smallbreak
Both \cite{WYZ} and \cite{ZZ} required a priori  the density to stay uniformly  bounded for all positive time. 
Our main goal here is to get rid of this assumption and to prove that the global solutions that have been constructed in \cite{DM-ripped}  
under much weaker regularity assumptions converge exponentially fast to the  constant state $(\bar\rho,0).$ 
 We shall also  specify the rate of convergence with respect to the viscosity coefficients, 
pointing out that the \emph{viscous effective flux} $G:=\nu\div u-P$ has faster decay.  
As a by-product of our results, we shall obtain an accurate control on the lower and upper bound of the density. In particular, no vacuum may appear if there is no vacuum initially. 
\medbreak
The rest of the paper is structured as follows. 
In the first section, as a warm-up, we prove exponential decay of solutions in a simpler situation, namely 
that of the inhomogeneous incompressible Navier-Stokes equations under similar assumptions on the data. 
The bulk of the paper is Section \ref{s:CNS}. There, we state different families of time decay estimates for (CNS), involving 
the energy and higher order quantities. 
Some key Poincar\'e type inequalities specific to the torus case are presented in Appendix.

\medbreak\noindent{\bf Notation.} Throughout the text,
$A\lesssim B$ means that  $A\leq CB$, where $C$ stands for various positive real numbers the value of which does not matter. 
 The notation $L_p$ will designate $L_p(\Omega,\dbar x)$ with $\dbar x$ the normalized Lebesgue measure on 
 $\T^d.$ The notation for the \emph{convective derivative}: 
$\t{u}\overset{\text{def}}{=}u_{t}+u\cdot \nabla u.$
 Finally, ${\mathbb P}$ (resp. $\Q$) denotes the Leray projector on divergence free (resp. potential) vector fields.


\section{The incompressible case}\label{s:INS}

Here  we are concerned with the  inhomogeneous incompressible Navier-Stokes system: 
 \begin{equation*}
\left\{\begin{aligned}
&\rho_{t}+\div(\rho u)=0    & \inn \quad \R_+\times \T^{d}, \\
 &(\rho u)_{t}+\div(\rho u\otimes u)-\mu\Delta u+\nabla P=0 & \inn \quad\R_+\times \T^{d}, \\
&\div u=0 & \inn \quad\R_+\times \T^{d},\\
&(\rho,u)|_{t=0}=(\rho_{0},u_{0}) & \inn \quad \T^{d},
\end{aligned}\right.\leqno(INS)
\end{equation*}
where $\rho, u$ stand for the density and velocity of the fluid, respectively, and $P$ is the (scalar) pressure function. 
Note that in contrast with the compressible situation, the pressure is not given. 
\medbreak
In a recent paper \cite{DM1},  P.B. Mucha and the first author established 
a global-in-time existence and uniqueness result for (INS) supplemented with any 
data  $(\rho_0,u_0)$ with $\rho_0$ just satisfying   \eqref{eq:inir} and  divergence free $u_0$ in $H^1$. 
 A `commercial'  version of the result obtained therein reads: 
\begin{theorem}\label{them1} Let $d=2,3.$ 
 Let $\rho_0$  satisfy \eqref{eq:inir} and let $u_0\in H^1(\T^d)$ be such that $\div u_0=0.$ 
 If $d=3,$ assume  additionally that  
 $$(\rho^*_0)^{3/2}\|\sqrt{\rho_0}u_0\|_{L_2}\|\nabla u_0\|_{L_2}\ll \mu^2.$$
  Then, System (INS) admits a unique global solution $(\rho, u, P)$ such that (among others)
\begin{equation}\label{eq:results1} 
\begin{aligned}
&\rho\in L_{\infty}(\R_+;L_\infty),\ u\in L_{\infty}(\R_+;H^1),\\
&\sqrt{\rho}u_t,\nabla^2 u,\nabla P\in L_2(\R_+;L_2),\ \nabla u\in L_{1,loc}(\R_+;L_\infty),\\
&\sqrt{\rho}u\in \cC(\R_+;L_2)\andf \rho \in \cC(\R_+;L_p) \quad\hbox{for all }\ 1\leq p<\infty. 
\end{aligned}
\end{equation}
\end{theorem}
In another recent paper \cite{DMP} in collaboration with P.B. Mucha and T. Piasecki, the first author established
that the solutions of the above theorem have exponential decay, if the fluid domain is smooth and bounded. 
As a warm-up, we shall establish that exponential decay also holds in  the torus case:
\begin{theorem}\label{lemed1}
Let $(\rho,u)$ be a solution to (INS) given  by Theorem \ref{them1}. 
There exists a positive constant $C_0$ depending only on the initial data and on the
measure of $\T^d,$ and $0<\beta_2<\beta_1$ such that for all $t\geq 0$, we have:
\begin{eqnarray}\label{eq:esul2}
&\Int_{\T^d} \rho(t)\abs{u(t)}^2\dbar x\leq e^{-\beta_1 t}\int_{\T^d} \rho_0\abs{u_0}^2\dbar x,\\\label{eq:esh1}
   &\underset{t\in \R_+}\sup  \|\nabla u(t)\|_{L_2}\leq C_0 e^{-\beta_2 t}, \\ \label{eq:esh2}
 &\Int_0^{\infty} e^{2\beta_2 t}\Bigl( \|\sqrt{\rho(t)} u_t(t)\|^2_{L_2}+ \|\nabla^2 u(t)\|^2_{L_2}+ \|\nabla P(t)\|^2_{L_2}\Bigr)\,dt\leq C_0.
 \end{eqnarray}
 Furthermore, there exists $\beta_3<\beta_1,$ we have
\begin{multline}\label{lemd2}
    \underset{t\in \R_+}\sup\|\sqrt{\rho t}\, e^{\beta_3t} u_t(t)\|_{L_2}^2+\underset{t\in \R_+}\sup
    \|\sqrt  t\,e^{\beta_3t}(\nabla^2u,\nabla P)(t)\|^2_{L_2}\\    
 +\int_0^\infty e^{2\beta_3t}\|\sqrt{t}\nabla u_t\|^2_{L_2}\,dt\leq C_0,
    \end{multline}
    Finally, there exists  $\beta_4<\beta_1,$ we have
    \begin{eqnarray}\label{eq:qwerrewq}
&e^{\beta_4 t}(\nabla^2u,\nabla P)\in L_1(\R_+;L_r(\T^d)) \quad\hbox{for all}\quad r<\infty,\\\label{eq:nul1li}
&e^{\beta_4 t}\nabla u\in L_1(\R_+;L_\infty(\T^d)).
\end{eqnarray}
    \end{theorem}
\begin{proof} 
What  differs from \cite{DMP} is the proof of \eqref{eq:esul2}, as 
the classical  Poincar\'e inequality that we would like to apply to  $z=u$:
 \begin{equation}\label{eq:poinca}\|z\|_{L_2}\leq c_{\T^d}\|\nabla z\|_{L_2},\end{equation}
 is valid only for functions  with mean value zero.
 As a  a substitute, we shall use  \eqref{eq:pti} (see the Appendix). 
\smallbreak 
Now, testing the momentum equation by $u,$ we get
$$\frac 12\frac{d}{dt} \int_{\T^d} \rho \abs{u}^2\dbar x+\mu \int_{\T^d} \abs{\nabla u}^2\dbar x=0.$$
Using   \eqref{eq:pti} and the fact that$\|\rho(t)\|_{L_2}$ is time independent, we get
$$\frac{c_{\T^d}^2\|\rho_0\|_{L_2}^2}{2\mu}\frac{d}{dt} \int_{\T^d} \rho \abs{u}^2\dbar x+\int_{\T^d} \abs{ u}^2\dbar x\leq 0.$$
Multiplying both sides by $\rho^*_0$ and using the obvious facts that 
$\rho \abs{ u}^2\leq \rho^*_0 \abs{ u}^2$  and that $\|\rho(t)\|_{L_\infty}$ is time independent,
we obtain
$$\frac1{\beta_1}\frac{d}{dt} \int_{\T^d} \rho \abs{u}^2\dbar x+\int_{\T^d} \rho\abs{ u}^2\dbar x\leq 0
\with  \beta_1:=\frac{2\mu}{\rho^{*}_0c_{\T^d}^2
\|\rho_0\|_{L_2}^2},$$
 whence  \eqref{eq:esul2}.
\smallbreak
 Proving  \eqref{eq:esh1} and \eqref{eq:esh2} from  \eqref{eq:esul2} is the same as for Lemma 6 in \cite{DMP}. 
 Combining these two inequalities with 
 the following Gagliardo-Nirenberg inequalities:
$$\begin{aligned}
\|\wt z\|_{L_{\infty}(\T^2)}&\lesssim \|\wt z\|^{\frac 12}_{L_{2}(\T^2)}\|\nabla^2 z\|^{\frac 12}_{L_{2}(\T^2)}&\andf
&\|\nabla z\|_{L_{3}(\T^2)}&\!\!\!\lesssim \|\nabla z\|^{\frac 23}_{L_{2}(\T^2)}\|\nabla^2 z\|^{\frac 13}_{L_{2}(\T^2)},\\
\|\wt z\|_{L_{\infty}(\T^3)}&\lesssim \|\nabla z\|^{\frac 12}_{L_{2}(\T^3)}\|\nabla^2 z\|^{\frac 12}_{L_{2}(\T^3)}
&\andf &\|\nabla z\|_{L_{3}(\T^3)}&\!\!\!\lesssim \| \nabla z\|^{\frac 12}_{L_{2}(\T^3)}\|\nabla^2 z\|^{\frac 12}_{L_{2}(\T^3)},
\end{aligned}$$
ensures in addition that 
\begin{equation}\label{eq:esother}\int_0^{\infty} e^{2\wt\beta_2 t}\Bigl( \| u\|^2_{L_\infty}+\|\nabla u\|^2_{L_3}\Bigr)\,dt
\leq C_0,\qquad \wt\beta_2<\beta_2.\end{equation}
Indeed, the above inequalities combined with \eqref{eq:esh1}, \eqref{eq:esh2}  (and \eqref{eq:pti} if $d=2$)  give 
$$\int_0^\infty e^{2\wt\beta_2t}\|\wt u\|^2_{L_\infty(\T^2)}\,dt\leq C_0,$$
and we  have $\|u\|_{L_{\infty}(\T^2)}\leq |\bar u|+\|\wt u\|_{L_{\infty}(\T^2)}$
with, owing to \eqref{eq:meanvalue} and to \eqref{eq:poinca},  
\begin{equation}
    \label{eq:aunu}
    |\bar u|\leq c_{\T^2}\|\rho_0-1\|_{L_2}\|\nabla u\|_{L_2}.
\end{equation}
\smallbreak 
Proving the second part of \eqref{eq:esother} is left to the reader.
\medbreak
For proving \eqref{lemd2},  it is only a matter of 
showing that \begin{equation}\label{eq:star0}\norm{e^{\beta t}\sqrt{t}\sqrt{\rho}{u}_{t}}_{L_{\infty}(L_{2})}+\norm{e^{\beta t}\sqrt{t}\nabla {u}_{t}}_{L_{2}(0,T\times \T^2)}\leq C_0.\end{equation}
Note that the momentum equation of $(INS)$ may be seen as the following Stokes system:
\begin{equation}\label{eq:stokescns}-\Delta u+\nabla P=\rho\dot u\andf \div u=0\quad\hbox{in }\ \R_+\times\T^2.\end{equation}
Hence using  
\begin{equation}\label{eq:interl4d}
\|\wt z\|_{L_{4}(\T^d)}\lesssim \|\wt z\|^{1-\frac d4}_{L_{2}(\T^d)}\|\nabla z\|^{\frac d4}_{L_{2}(\T^d)},\qquad d=2,3,
\end{equation}
 implies  that 
$$\begin{aligned}
\norm{e^{\beta t}\sqrt{t}\nabla^{2}u}_{L_{2}}\!+\!\norm{e^{\beta t}\sqrt{t}\nabla P}_{L_{2}}&\leq C(\norm{e^{\beta t}\sqrt{t}\sqrt{\rho}u_t}_{L_{2}}\!+\!\norm{e^{\beta t}\sqrt{t}\sqrt{\rho}u\cdot \nabla u}_{L_{2}})\\
&\leq C(\norm{e^{\beta t}\sqrt{t}\sqrt{\rho}u_t}_{L_{2}}\!+\!\sqrt{\rho_0^*}e^{\beta t}\sqrt{t}\|u\|_{L_4}\| \nabla u\|^{1-\frac d4}_{L_{2}}\|\nabla^2 u\|^{\frac d4}_{L_2})
\end{aligned}
$$
whence, by Young inequality and  \eqref{eq:esh1}, \eqref{eq:aunu}, \eqref{eq:interl4d},
\begin{equation}\label{eq:dstar}
\|e^{\beta t}\sqrt{t}(\nabla^{2}u,\nabla P)\|_{L_{\infty}(L_2)}\lesssim C_0(\rho_0^*)^{\frac{2}{4-d}}\sqrt{t}e^{(\beta-2\beta_2)t}+ \norm{e^{\beta t}\sqrt{t}\sqrt{\rho}u_t}_{L_{\infty}(L_2)}.
\end{equation}
To prove \eqref{eq:star0},   apply $\d_{t}$ to the momentum equation of (INS). We get
 $$   \rho u_{tt}+\rho u\cdot \nabla u_{t}-\Delta u_{t}+\nabla P_{t}=-\rho_{t}\t{u}-\rho u_{t}\cdot \nabla u.$$
By taking the $L_{2}$ scalar product of the above equation with  $e^{2\beta t}u_t,$ we obtain
$$\displaylines{
\quad
\frac{1}{2}\frac{d}{dt}\int_{\T^{d}}\rho e^{2\beta t} t\abs{u_{t}}^{2}\dbar x+\int_{\T^{d}}e^{2\beta t} t\abs{\nabla u_{t}}^{2}\dbar x\leq \beta\int_{\T^{2}}e^{2\beta t} t\rho\abs{u_{t}}^{2}\dbar x+\frac 12 \int_{\T^{d}}e^{2\beta t} \rho\abs{u_{t}}^{2}\dbar x
\hfill\cr\hfill
-\int_{\T^{d}}e^{2\beta t} t\rho_{t}\t{u}\cdot u_{t}\dbar x-\int_{\T^{d}}e^{2\beta t} t\rho (u_{t}\cdot \nabla u)\cdot u_{t}\dbar x=\sum_{i=1}^4 I_{i}.}$$
For term $I_1,$ using  \eqref{eq:utl2}, \eqref{eq:esul2} and \eqref{eq:esh1}    gives:
$$
\begin{aligned}
I_1\leq &\beta \rho^*_0 te^{2\beta t}\|u_t\|^2_{L_2}\\
\leq &C\beta \rho^*_0 te^{2\beta t}((\rho^*_0)^2\|u\|^2_{L_2}\|\nabla u\|^2_{L_2}+c_{\T^d}^2\|\nabla u_t\|^2_{L_2})\\
\leq & C\beta \rho^*_0 c_{\T^d}^2   te^{2\beta t}\|\nabla u_t\|^2_{L_2}+ C^2_0\beta (\rho^*_0)^3 te^{2(\beta-2\beta_2) t}\cdotp
\end{aligned}
$$
For term $I_{3},$ the mass equation of $(INS)$ and integration by parts
yield
\begin{multline*}
    I_{3}=\int_{\T^{d}}e^{2\beta t} t \div(\rho u)\t{u}\cdot u_{\tau}\dbar x
     =I_{31}+I_{32}\\
     \with I_{31}:=-\int_{\T^{d}}e^{2\beta t} t (\rho u\cdot \nabla{\t{u}}) \cdot u_{\tau}\dbar x\andf I_{32}:=
    -\int_{\T^{d}}e^{2\beta t} t (\rho u\cdot \nabla{u_{\tau})\cdot  \t{u} }\dbar x.\end{multline*}
 Since $\nabla\t{u}=\nabla u_{t}+\nabla(u\cdot \nabla u)$, we may write
\begin{equation*}
    \begin{aligned}
    I_{31}
    =&-\int_{\T^{d}}e^{2\beta t} t (\rho u\cdot \nabla u_{t}) \cdot u_{t}\dbar x
    -\int_{\T^{d}}e^{2\beta t} t (\rho u\cdot  u\cdot \nabla^{2} u) \cdot u_{t}\dbar x\\
    &-\int_{\T^{d}}e^{2\beta t} t (\rho u\cdot \nabla u\cdot \nabla u) \cdot u_{t}\dbar x=:I_{311}+I_{312}+I_{313}
    \end{aligned}    \end{equation*}
    and 
    $$ I_{32}=-\int_{\T^{d}} e^{2\beta t} t(\rho u\cdot \nabla u_{t})\cdot u_{t}\dbar x-
   \int_{\T^{d}} e^{2\beta t} t(\rho u\cdot \nabla u_t)\cdot ( u\cdot \nabla u)\dbar x.$$
Before bounding  $I_{31}$, we point out a useful inequality that stems from  \eqref{eq:utl2} and embedding $H^1(\T^d)\hookrightarrow L_6(\T^d)$:
\begin{equation}\label{eq:esutl6}
\|e^{\beta t}\sqrt{t}\sqrt{\rho}u_t\|_{L_6(\T^d)}\leq e^{\beta t}\sqrt{t}\sqrt{\rho^*_0}(\rho^*_0\|u\|_{L_2(\T^d)}\|\nabla u\|_{L_2(\T^d)}+c_{\T^d}\|\nabla u_{t}\|_{L_2(\T^d)}).
\end{equation}
Now, by H\"older inequality, we have
$$I_{311}\leq \sqrt{\rho^*_0}\|\sqrt{t}e^{\beta t}\nabla u_t\|_{L_2(\T^d)}\|\sqrt{\rho t}e^{\beta t} u_t\|_{L_2(\T^d)}\|u\|_{L_\infty(\T^d)}.$$
For term $I_{312}$, applying \eqref{eq:esutl6} yields
\begin{equation*}
\begin{aligned}
I_{312}
\leq & \sqrt{\rho^*_0}\|\nabla^2 u\|_{L_2(\T^d)}\|u\|^2_{L_6(\T^d)}\|e^{\beta t}\sqrt{t}\sqrt{\rho}u_t\|_{L_6(\T^d)}\sqrt{t}e^{\beta t}\\
\lesssim& \|\nabla^2 u\|_{L_2(\T^d)}\|u\|^2_{L_6(\T^d)}\sqrt{t}e^{\beta t}(e^{\beta t}\sqrt{t}\|u\|_{L_2(\T^d)}\|\nabla u\|_{L_2(\T^d)}+\|e^{\beta t}\sqrt{t}\nabla u_{t}\|_{L_2(\T^d)})\cdotp
\end{aligned}
\end{equation*}
Let us just treat $I_{313}$ in the case $d=3$ (the case $d=2$ is left to the reader). 
Then,  taking advantage of Sobolev embedding $\H^1(\T^3)\hookrightarrow L_6(\T^3)$ and Young inequality, we get
$$
\begin{aligned}
I_{313}
\leq & \sqrt{\rho^*_0}\|\nabla u\|^2_{L_6}\|u\|_{L_6}\|e^{\beta t}\sqrt{t}\sqrt{\rho}u_t\|_{L_2}\sqrt{t}e^{\beta t}\\
\lesssim &\|e^{\beta t}\sqrt{t}\sqrt{\rho}u_t\|^2_{L_2}\|\nabla^2 u\|^2_{L_2}+e^{2\beta t}t\|\nabla^2 u\|^2_{L_2}\|\nabla u\|^2_{L_2} \cdotp
\end{aligned}
$$
Sobolev embedding directly implies  that
\begin{equation*}
 \begin{aligned}
   I_{31} &\leq \eps \|e^{\beta t}\sqrt{t}\nabla u_{t}\|^2_{L_2(\T^3)}+Cte^{2\beta t}\|\nabla^2 u\|_{L_2}\|\nabla u\|^2_{L_2}\|u\|_{L_2}\|\nabla u\|_{L_2} \\
    &\qquad\qquad+C\|e^{\beta t}\sqrt{t}\sqrt{\rho}u_t\|^2_{L_2}(\|u\|^2_{L_\infty}+\|\nabla^2 u\|^2_{L_2})\\
    &\qquad\qquad+Cte^{2\beta t} \|\nabla^2 u\|^2_{L_2}\|\nabla u\|^2_{L_2}(1+\|\nabla u\|^2_{L_2})\\
    &\leq \eps \|e^{\beta t}\sqrt{t}\nabla u_{t}\|^2_{L_2(\T^3)}+C_0^4te^{2(\beta-2\beta_2) t}\|\nabla^2 u\|_{L_2} \\
    &\qquad\qquad+C\|e^{\beta t}\sqrt{t}\sqrt{\rho}u_t\|^2_{L_2}(\|u\|^2_{L_\infty}+\|\nabla^2 u\|^2_{L_2})\\
    &\qquad\qquad+C_0^2te^{2(\beta-\beta_2) t}(1+e^{-2\beta_2t}) \|\nabla^2 u\|^2_{L_2}.
    \end{aligned}
\end{equation*}
For   $I_{32},$ embedding $H^1(\T^d)\hookrightarrow L_6(\T^d)$ ensures that for all $\eps>0,$
$$
\begin{aligned}
 I_{32} &\leq \sqrt{\rho^*_0}\|e^{\beta t}\sqrt{t} \nabla u_t\|_{L_2}(\|u\|_{L_{\infty}}\|e^{\beta t}\sqrt{t}\sqrt{\rho}u_t\|_{L_2}+\sqrt{\rho^*_0}e^{\beta t}\sqrt{t}\|\nabla u\|_{L_6}\|u\|^2_{L_6})\\
 &\lesssim  \eps \|e^{\beta t}\sqrt{t}\nabla u_{t}\|^2_{L_2(\T^2)}+C(\|u\|^2_{L_\infty}\|e^{\beta t}\sqrt{t}\sqrt{\rho}u_t\|^2_{L_2}+C_0^2te^{2(\beta-2\beta_2)t}(1+\|\nabla^2 u\|^2_{L_{2}})).
\end{aligned}  
$$
For $I_{4}$, thanks to \eqref{eq:esutl6},  \eqref{eq:esh1} and Young inequality, one has
$$ \begin{aligned}
I_{4}
&\leq \|e^{\beta t}\sqrt{t}\sqrt{\rho}u_t\|_{L_2}\|\nabla u\|_{L_3}\|e^{\beta t}\sqrt{t}\sqrt{\rho}u_t\|_{L_6}\\
  &\leq  \eps\|e^{\beta t}\sqrt{t}\nabla u_{t}\|^2_{L_2}+C \rho^*_0\bigl(\|e^{\beta t}\sqrt{t}\sqrt{\rho}u_t\|^2_{L_2}\|\nabla u\|^2_{L_3}+te^{2\beta t} (\rho^*_0)^2\|u\|^2_{L_2}\|\nabla u\|^2_{L_2}\bigr)\\
&\leq  \eps\|e^{\beta t}\sqrt{t}\nabla u_{t}\|^2_{L_2}+C \rho^*_0(e^{-2\beta_2 t}\|e^{\beta t}\sqrt{t}\sqrt{\rho}u_t\|^2_{L_2}\|e^{\beta_2 t}\nabla u\|^2_{L_3}+C_0^2te^{2(\beta-2\beta_2) t} (\rho^*_0)^2).    \end{aligned}$$
Putting  together with \eqref{eq:esul2}, \eqref{eq:esh1} and  choosing $\eps$ small enough 
we conclude that for all  $\beta\leq \beta_2$ there exists a positive  constant $c$ only depending on $\beta$ and $\beta_2$ such that
\begin{multline*}
\frac{d}{dt}\int_{\T^{d}}\rho e^{2\beta t} t\abs{u_{t}}^{2}\dbar x+\int_{\T^{d}}e^{2\beta t} t\abs{\nabla u_{t}}^{2}\dbar x
\leq C\|\sqrt{\rho}e^{\beta t} u_t\|_{L_2}^2\\+C_0^2t e^{-ct} (1+\|\nabla^2 u\|^2_{L_2}+\|u\|^2_{L_{\infty}})
+C\|e^{\beta t}\sqrt{t}\sqrt{\rho}u_t\|^2_{L_2}(\|u\|^2_{L_\infty}+\|\nabla^2 u\|^2_{L_2}+\|\nabla u\|^2_{L_3}).\end{multline*}
Hence, using Gronwall inequality and \eqref{eq:esh2}, \eqref{eq:esother} ensures \eqref{eq:star0} and thus \eqref{lemd2}.
\smallbreak
Granted with this inequality, it is easy to establish the last two inequalities of the theorem. 
Clearly, the second one   follows from the first one with $r>d,$
Inequality \eqref{eq:esh1} and  the following Gagliardo-Nirenberg inequality: 
$$\|\nabla z\|_{L_\infty}\leq C\|\nabla z\|^\theta_{L_2}\|\nabla^2 z\|^{1-\theta}_{L_r} \andf \theta=\frac{2(r-d)}{2(r-d)+rd}\cdotp $$ 
In order to prove \eqref{eq:qwerrewq}, we observe that for $2<r<\infty,$ we have 
\begin{equation}\label{eq:gnlr}
\|z\|_{L_r}\leq C \|z\|^{\gamma}_{L_2}\|\nabla z\|^{1-\gamma}_{L_2}\with \gamma=1+\frac dr-\frac d2\cdotp
\end{equation} 
Now, since $(u,\nabla P)$ satisfies \eqref{eq:stokescns}, 
using the $L_r$ regularity theory for the Stokes system and \eqref{eq:gnlr} implies  that
$$\int_0^\infty e^{\beta t}
\|(\nabla^2 u,\nabla P)\|_{L_r}\leq C \sqrt{\rho^*_0}\int_0^\infty e^{\beta t}\bigl(
\|\sqrt{\rho}u_t\|_{L_r}+ \sqrt{\rho^*_0}\|u\|_{L_\infty}\| \nabla u\|^\gamma_{L_2}\|\nabla^{2}u\|^{1-\gamma}_{L_2}\,dt\bigr)\cdotp$$
On the one hand, by using \eqref{eq:gnlr}, for all $\beta<\wt\beta_2,$ the second term may be bounded by means of 
\eqref{eq:esh1}, \eqref{eq:esh2}, \eqref{eq:esother}. 
On the other hand, owing to   \eqref{eq:gnlr} and \eqref{lemd2}, we have for all $\beta<\beta_3,$
$$\begin{aligned}
\int_0^{\infty}e^{\beta t}\|\sqrt{\rho}&u_t\|_{L_r}\,dt
= \int_0^\infty \frac{1}{t^{\frac{1-\gamma}{2}}} e^{(\beta-\gamma \beta_2-(1-\gamma)\beta_3) t}\|e^{\beta_2 t}\sqrt{\rho}u_t\|^{\gamma}_{L_2}\|e^{\beta_3 t}\sqrt{t}\nabla u_t\|^{1-\gamma}_{L_2}\,dt\\
&\lesssim \biggl(\int_0^1 \frac{dt}{t^{1-\gamma}}+\int_1^{\infty} e^{2(\beta-\beta_3) t}\,dt\biggr)^{\frac 12} \|e^{\beta_2 t}\sqrt{\rho}u_t\|^{\gamma}_{L_2(\R_+;L_2)}\|e^{\beta_3 t}\sqrt{t}\nabla u_t\|^{1-\gamma}_{L_2(\R_+;L_2)}\\
&\leq C_0.
\end{aligned}$$
Putting together the above inequalities completes the proof of the theorem. 
\end{proof}
\begin{remark}
 As in \cite{DMP}, we expect the above results to imply stability with respect to the data. 
\end{remark}


\section{The compressible case}\label{s:CNS}

As a starting point of our study of (CNS), let us recall the result of \cite{DM-ripped}. 
 \begin{theorem}\label{themc1}
 Consider any nonnegative bounded function $\rho_0$ fulfilling 
\begin{equation}\label{eq:inrCNS} 
\rho_{0,*}:=\underset{x\in\T^2}{\inf}\rho_0\geq0\andf \rho^*_0:=\underset{x\in\T^2}{\sup}\rho_0<\infty,
\end{equation}
and  vector field $u_0$ in $H^1(\T^2)$ satisfying for some fixed $K>0,$
\begin{equation}\label{eq:div}
\|\div u_0\|_{L_2}\leq K\nu^{-1/2}\andf\int_{\T^2} \rho_0 u_0\dbar x=0.\end{equation} 
Let us define the \emph{viscous effective flux} $G$ by the relation:
\begin{equation}\label{def:G} G:=\nu \div u-P.\end{equation}
There exists a positive number $\nu_0$ depending only on $\T^2,$  $\kappa,$ $K,$ $\gamma,$ $\mu,$ $\|\sqrt{\rho_0}u_0\|_{L_2},$ $\|\nabla u_0\|_{L_2}$ and $\|\rho_0\|_{L_\infty}$ such that if $\nu>\nu_0$ then System (CNS)  with pressure law \eqref{eq:isentro} admits a global-in-time solution $(\rho,u)$ fulfilling
$$    \begin{aligned}
    &\rho \in L_\infty(\R_+\times \T^2)\cap \cC(\R_+;L_p(\T^2)) \ \hbox{ for all }\  p<\infty \andf \sqrt{\rho}u\in \cC(\R_+;L_2(\T^2)),\\
    & u\in L_\infty(\R_+;H^1(\T^2)),\ (\nabla^2 \p u,\nabla G,\sqrt{\rho} \t u)\in L_2(\R_+\times \T^2),\\
    &\sqrt{\rho t} \t u \in L_{\infty,loc}(\R_+;L_2(\T^2))\andf \sqrt{t}\nabla \t u\in L_{2, loc}(\R_+;L_2(\T^2)).
    \end{aligned}$$
In addition, both $\div u$ and ${\rm curl}\, u$ are in $L_{r,loc}(\R_+;L_\infty(\T^2))$ for all $r<2,$
 and \eqref{conlaws}, \eqref{eq:energy} are satisfied  on $\R_+.$  
\end{theorem}

  Our main goal  is to prove that the solutions constructed in the above theorem converge exponentially fast to 
   $(\bar\rho,0),$ with $\bar\rho$ being the mean value of $\rho_0$ on $\T^2,$  when the time goes to infinity.
In this direction, let us state a first result. 
   \begin{theorem}\label{them:edcns}
      Let $(\rho,u)$ be a solution to (CNS) given  by Theorem \ref{themc1}.  
      Set 
     \begin{equation}\label{def:tilde}
     \wt G:=G-\bar G\andf \wt P:=P-\bar P\with \bar G:=\int_{\T^2} G\dbar x \andf  \bar P:=\int_{\T^2} P\dbar x.
    \end{equation} 
      There exist two  positive constants $C_0$  and $\nu_0$ (depending on $E_0,$ $\|\nabla u_0\|_{L_2},$ 
 $\rho^*_0$,   $\mu,$ $\kappa,$ $\T^2$ and $\gamma,$ but independent of $t$),  $\alpha_1$ depending only  on  $\mu,$ $\T^2,$ and $\rho^*_0,$ and $c_{\rho_0^*}$ depending only on $\T^2,$ $\rho^*_0,$ $\kappa$  and $\gamma,$  such that if $\nu>\nu_0,$ then we have  for all  $t\geq 0$:
\begin{enumerate}
    \item Exponential decay of the energy:
    \begin{equation}\label{eq:energy0} E(t)\leq 2E_0 e^{-\alpha_0 t}\with
     \alpha_0:=c_{\rho_0^*}
\min\biggl(\frac{\mu}{\bar\rho},\frac{\bar\rho   P'(\bar\rho)}\nu\biggr),
  \andf      \end{equation}
        \begin{equation} \label{eq:energy1}\qquad D(t)\leq 2D(0) e^{-\alpha_1t} + C_0\nu^{-1}e^{-\alpha_0 t}\with D:=\frac{1}{2} \int_{\T^2} \rho \abs{u}^2\dbar x
    +\frac{1}{\nu}\int_{\T^2} e\dbar x.\end{equation}
    \item Exponential decay of $H^1$ norms and of the viscous effective flux: 
\begin{equation}\label{eq:decayH1}   {\nu}\|\div u(t)\|_{L_2}^2+\|\nabla \p u(t)\|_{L_2}^2+\nu^{-1}\|\rG(t)\|_{L_2}^2
\leq  C_{0} \Bigl(e^{-\frac{\alpha_1}2 t}+\nu^{-1}e^{-\frac{\alpha_0}2t}\Bigr),\quad t\geq0.\end{equation}
    \item  Exponential decay of $L_2$-in-time norms:
    \begin{equation}\label{eq:decayL2t}\int^{\infty}_0 e^{\frac{\alpha_0}4 t}\biggl(\int_{\T^2}\bigl(\rho \abs{u_t}^2+ \abs{\div u}^2+\mu^2 \abs{\nabla^2 \p u}^2
        +\abs{\nabla G}^2\bigr)\dbar x+\nu^{-1}|\bar P(t)-1|^2\biggr)dt\leq C_0.\end{equation}
            \item Exponential decay of the pressure: 
   \begin{equation}\label{eq:decayP}   \|\wt P(t)\|_{L_2}\leq C_0 e^{-\frac{\alpha_0}4t}.\end{equation}
    \item Lower  and upper  bound of the density: 
    \begin{equation}\label{eq:rhobound}  \frac12\rho_{0,*}\leq \rho(t,x)    \leq C_0 \quad\hbox{for a.e. } (t,x)\in\R_+\times\T^2.\end{equation}
\end{enumerate}
\end{theorem}
\begin{remark} Looking at the definition of $\alpha_0,$ the reader may find  it odd that 
the rate of decay deteriorates when $\nu$ increases.   This may be already observed on 
the linearized compressible Navier-Stokes about the constant
state $(\bar\rho, 0).$ Assuming for simplicity that $\bar\rho=1$ and  $P'(\bar\rho)=1$ and that 
$\T^2$ is the torus with lengths $2\pi$ in the two directions,  the eigenvalues corresponding to the frequency $k\in\Z^2$ are:
\begin{itemize}
\item  $\mu|k|^2$ for $\p u$;\smallbreak
\item
$\lambda^\pm_k:= -\frac{\nu |k|^2}{2}(1\pm R_k)$ with $R_k:=\sqrt{1-\frac{4}{\nu^2 |k|^2}}$   for the system satisfied by $(a,\div u).$
\end{itemize}
In the asymptotics $\nu\to\infty$ then $\lambda_k^-\to -1/\nu$ so 
that the overall decay of the solution is only $e^{-\frac t\nu}.$
At the same time, $G$ tends to be a good approximation of
the parabolic mode with diffusion $\nu,$ and it is thus expected that it decays faster. 
This has been indeed partially justified in the small data case (see \cite{DM-adv}). 
Whether it is still true in our context of large data with vacuum is an open question.
\end{remark}

Our next result states that the time weighted quantities appearing at the end of Theorem \ref{themc1}
also have exponential decay.
\begin{theorem} \label{them:edcns2} Under the assumptions of Theorem \ref{them:edcns} we have in addition
\begin{multline*}
\underset{t\in \R_+}{\sup} t\bigl(e^{\alpha_0 t}\|\sqrt \rho\, \t u\|_{L_2}^2+\|\nabla{\mathbb P}u(t)\|_{L_2}^2
+\nu^{-1}\|G(t)\|_{L_2}^2\bigr)
\\+\int_0^\infty te^{\alpha_0 t}\bigl(\|\nabla G\|_{L_2}^2+\|\nabla^2 {\mathbb P} u\|_{L_2}^2+
\|\nabla \t u\|_{L_2}^2+\|\div\t u\|_{L_2}^2\bigr)dt\leq C_{0}\end{multline*}
with  $\alpha_0$  defined as in Theorem \ref{them:edcns}
(but, possibly, smaller). 

Furthermore,  for all small enough $\theta>0,$  we have
$$\int_0^\infty e^{{\alpha_0}t}\|\rG\|_{L_\infty}\,dt\leq C_{0}\,\nu^{\theta}\andf\int_0^\infty e^{{\alpha_0}t}\|\nabla \p u\|_{L_\infty}\,dt\leq C_{0}.$$
\end{theorem}
The rest of this section will be dedicated to proving the above results.


\subsection{Proving  Theorem \ref{them:edcns}}

 Throughout the proof, we shall denote by $C_0$ (resp. $C_{0,\rho^*}$) various  `constants' that only depend on  $E_0,$ $\|\nabla u_0\|_{L_2},$  $\mu,$ $\kappa,$ $\T^d,$ $\gamma$ and  $\rho^*_0$ (resp. $\rho^*$).  
 We reserve the notation $C_{\rho^*}$ for constants depending only on $\rho^*,$ $\gamma$ and $\T^2.$ 

\smallbreak

The first step of  the proof is to establish \eqref{eq:energy0} with $\rho^*$ instead of $\rho^*_0,$
assuming that  \begin{equation}\label{eq:suprho}\rho^*:= \underset{(t,x)\in \R_+\times\T^2}{\sup}\rho(t,x)<\infty.\end{equation}
As in \cite{WYZ}, this will be achieved by introducing  a modified energy functional,  equivalent to the classical energy. 
Once \eqref{eq:energy0} (and its variation \eqref{eq:energy1}) would have been proved, we will  
plug the exponential decay information in the higher order energy functional defined in \cite{DM-ripped}, so as to 
establish Inequalities \eqref{eq:decayH1}, \eqref{eq:decayL2t} and \eqref{eq:decayP}. 
This exponential decay will also enable us to get  a precise control of the lower bound of the density
(see  \eqref{eq:rhobound}). 

In the first four steps, we shall prove Inequalities \eqref{eq:energy0}, \eqref{eq:energy1}, \eqref{eq:decayH1}, 
\eqref{eq:decayL2t}, \eqref{eq:decayP} with $C_{0,\rho^*}$ instead of $C_0,$ under Assumption \eqref{eq:suprho}.
In the fifth step, these inequalities, combined with a bootstrap argument, will be used to prove \eqref{eq:rhobound}. 
Thanks to that, we will be able to  get rid  of the dependence on  $\rho^*$ in  \eqref{eq:energy0}, \eqref{eq:energy1}, \eqref{eq:decayH1}, \eqref{eq:decayL2t} and \eqref{eq:decayP}.
\smallbreak

In order to simplify the presentation, we perform the change of unknowns
\begin{equation}\label{eq:rescaling}
\rho(t,x)= \bar\rho\wt\rho\Bigl(\frac{\bar\rho}\mu t,\frac{\bar\rho\sqrt{P'(\bar\rho)}}\mu x\Bigr)
\andf 
u(t,x)=\sqrt{\frac{P(\bar\rho)}{\bar\rho}}\wt u\Bigl(\frac{\bar\rho}\mu t,\frac{\bar\rho\sqrt{P'(\bar\rho)}}\mu x\Bigr)\cdotp
\end{equation}
Note that $(\rho,u)$ satisfies (CNS) in $\T^2$  with viscosity coefficients
$\lambda$ and $\mu$  if and only if  $(\wt\rho,\wt u)$ satisfies (CNS) 
with viscosity coefficients $\lambda/\mu$ and $1,$
pressure law $\wt P(r)=\wt r^\gamma,$  in the torus $\wt\T^2:=\frac{\bar\rho\sqrt{P'(\bar\rho)}}\mu  \T^2,$
and the mean value of the density  is $1.$
Therefore, in the rest of the proof, we will focus on the case
 $\kappa=\mu=\bar\rho=1$ and remove all the tildes for notational  simplicity.
 We shall use repeatedly that \begin{equation}\label{eq:assum}
\int_{\T^d} \rho(x)\dbar x=1 \andf \int_{\T^d} (\rho u)(x)\dbar x=0.
\end{equation}

\subsubsection*{Step 1: Exponential decay of the energy functional}

Our starting point is the energy balance \eqref{eq:energy} \emph{before time integration}, namely: 
$$\frac{d}{dt}E(t)+ \int_{\T^2} \abs{\nabla u}^2\dbar x+(\lambda+1)\int_{\T^2} \abs{\div u}^2\dbar x=0.$$
Since 
\begin{equation}\label{eq:gu}
     \|\nabla u\|^2_{L_2}+(\lambda+1)\|\div u\|^2_{L_2}=\|\nabla \p u \|^2_{L_2}+\nu \|\div u\|^2_{L_2}\end{equation}
and  $\nu=\lambda+2,$ we have 
\begin{equation}\label{eq:eqedl2}
    \frac{d}{dt}E(t)+ \int_{\T^2} \abs{\nabla \p u}^2\dbar x+\nu\int_{\T^2} \abs{\div u}^2\dbar x=0.
\end{equation}
Let $a:=\rho-1.$ As $P=(1+a)^\gamma$ with $\gamma\geq1,$ the following 
inequality (that can be proved by considering
separately the cases $a\in[-1,0]$ and $a\geq0$) is true:
$$aP\geq a+a^2.$$
As   the  mean value  of $a$ is $0,$ we  thus get: 
\begin{equation}\label{eq:reap}
\int_{\T^2} a^2\dbar x \leq \int_{\T^2} a P\dbar x.
\end{equation}
Applying the operator $(-\Delta)^{-1}\div $ to the momentum equation yields
\begin{equation}\label{eq:p}
P=(-\Delta)^{-1}\div \bigl((\rho u)_{t}+\div(\rho u\otimes u)\bigr)+\nu\div u.
\end{equation}  
Hence, thanks to the first equation of $(CNS),$ Cauchy-Schwarz inequality
and the properties of continuity of $(-\Delta)^{-1}\nabla^2$ in $L_2,$ we obtain
$$
\begin{aligned}
\int_{\T^2} aP\dbar x&=\int_{\T^2} (-\Delta)^{-1}\div \bigl((\rho u)_{t}+\div(\rho u\otimes u)\bigr)\cdot a\dbar x+\nu\int_{\T^2} a\,\div u\dbar x\\
&=\frac{d}{dt}\int_{\T^2}  (-\Delta)^{-1}\div (\rho u)\cdot a\dbar x-\int_{\T^2} \nabla(-\Delta)^{-1}\div(\rho u)\cdot \rho u\dbar x\\
&\qquad\qquad+\int_{\T^2} (-\Delta)^{-1}\div(\div(\rho u\otimes u))\cdot a\dbar x+\nu\int_{\T^2} a\,\div u\dbar x\\
\leq \frac{d}{dt}&\int_{\T^2}(-\Delta)^{-1}\div (\rho u)\cdot a\dbar x+\|{\rho}u\|^2_{L_2}+\|\rho u\otimes u\|_{L_2}\|a\|_{L_2}+\nu\|\div u\|_{L_2}\|a\|_{L_2}.
\end{aligned}$$
Along with \eqref{eq:reap} and Young inequality, the above inequality implies that
$$\begin{aligned}
-\frac{1}{\nu}\frac{d}{dt}\int_{\T^2}(-\Delta)^{-1}\div (\rho u)\cdot a&\dbar x+\frac{1}{\nu}\int_{\T^2} a^2\dbar x
\\&\leq \frac{\rho^*}{\nu}\|\sqrt{\rho}u\|^2_{L_2}+\frac{1}{\nu}\|\rho u\otimes u\|^2_{L_2}+{\nu}\|\div u\|^2_{L_2}+\frac{1}{2\nu}\|a\|^2_{L_2}.\end{aligned}$$
Now, combining  H\"older inequality
and  Sobolev inequality \eqref{eq:sobolev} with $p=6$ yields
\begin{align}
    \|\rho^{1/4} u\otimes u\|^2_{L_2}&\leq \|\sqrt{\rho}u\|_{L_2} 
    \|u\|^3_{L_6}\nonumber\\\label{eq:uL4}
    &\leq C c_{\T^2}\|\sqrt{\rho}u\|_{L_2}(\rho^*)^{3}\|\nabla u\|^3_{L_2}\cdotp
\end{align}  
Hence,  for some suitable  $C_{\rho^*}$ depending continuously on $\rho^*,$ 
$$-\frac{1}{\nu}\frac{d}{dt}\int_{\T^2}(-\Delta)^{-1}\div (\rho u)\cdot a\dbar x+
\frac{1}{2\nu}\|a\|_{L_2}^2\leq \frac{2\rho^*}\nu\|\sqrt\rho u\|_{L_2}^2
+C_{\rho^*}\frac{c_{\T^2}^{2/3}}\nu\|\nabla u\|_{L_2}^6+\nu\|\div u\|^2_{L_2}.
$$
Since  $\|\nabla u\|_{L_2}$ is bounded uniformly in time in terms of the initial data (see \cite[Proposition 2.1]{DM-ripped}), thanks to \eqref{eq:poincare} and  \eqref{eq:assum}, we conclude that
\begin{equation}\label{eq:eqaimp}
-\frac{1}{\nu}\frac{d}{dt}\int_{\T^2}(-\Delta)^{-1}\div (\rho u)\cdot a\dbar x+\frac{1}{2\nu}\|a\|_{L_2}^2\leq c_{\T^2}\frac{C_0}{\nu}\|\nabla u\|^2_{L_2}+{\nu}\|\div u\|^2_{L_2}.
\end{equation}
Consequently, if we set for some $\eta\in(0,1)$ (that will be completely 
fixed at the end of the proof): 
\begin{equation}\label{def:tildeE}
\check E:= E -\frac{\eta}\nu\int_{\T^2}(-\Delta)^{-1}\div (\rho u)\cdot a\dbar x,
\end{equation}
then we get, thanks to \eqref{eq:eqedl2}, 
\begin{equation}\label{eq:tildeE}
\frac d{dt}\check E +\|\nabla \p u\|_{L_2}^2+\nu(1-\eta)\|\div u\|_{L_2}^2
+\frac\eta{2\nu}\|a\|_{L_2}^2\leq {C_0\eta}c_{\T^2}^2\nu^{-1}\|\nabla u\|_{L_2}^2.
\end{equation}
As  the mean value of $a$ is $0,$  owing to Poincar\'e inequality and to Lemma
\ref{lem:eseq}, we have
\begin{align}\label{eq:equaaru}
    \frac\eta\nu\biggl|\int_{\T^2}(-\Delta)^{-1}\div (\rho u)\cdot a\dbar x\biggr|
    &\leq  \frac{\eta}\nu\|\Delta^{-1}\div (\rho u)\|_{\H^1}\| a\|_{\H^{-1}}\nonumber\\
    &\leq    \frac{\eta}\nu\sqrt{\rho^*}\|\sqrt{\rho}u\|_{L_2}\|a\|_{L_2}\nonumber\\
    &\leq \frac{\eta^2}{\nu^2}\frac{C_{\rho^*}^2}2\|\sqrt{\rho}u\|_{L_2}^2+\frac 12\|e\|_{L_1}.
\end{align}
Therefore, choosing  $\eta\leq\nu/(4C_{\rho^*})$
implies  $\check E\simeq E.$ Furthermore, since 
$$\|\nabla u\|^2_{L_2}= \|\nabla \p u \|^2_{L_2}+\|\div u\|^2_{L_2},$$
it is not difficult to see from \eqref{eq:tildeE} that 
if, additionally $\eta {c_{\T^2}^{2}}\leq  c_0\nu$  (with $c_0$
depending only on the data and on $\rho^*$), 
and $\nu\geq1$ then 
$$\frac d{dt}\check E +\frac14\|\nabla u\|_{L_2}^2+\frac{\eta}{2\nu}\|a\|_{L_2}^2\leq 0.
$$
Finally, by using Proposition \ref{prop1} and  the obvious fact that $\rho \abs{ u}^2\leq 
\rho^* \abs{ u}^2, $ we obtain
\begin{equation}\label{eq:poincareru}
\frac{1}{2{c_{\T^2}^{2}}(\rho^*)^3}\int_{\T^2} \frac{\rho\abs{u}^2}2\dbar x\leq \frac14\int_{\T^2} \abs{\nabla u}^2\dbar x,
\end{equation} 
while Lemma   \ref{lem:eseq} yields
$$\frac{\eta}{2\nu}\|a\|_{L_2}^2\geq \frac{\eta}{2C_{\rho^*}\nu}\|e\|_{L_1}.$$
Hence we end up with, using $\check E\simeq E,$  
 \begin{equation}\label{eq:ieecE} 
\frac{d}{dt}\check E+\frac14\min\biggl(\frac{1}{(\rho^*)^3{c_{\T^2}^{2}}},\frac{\eta}{C_{\rho^*}\nu}\biggr)\check E\leq 0.
 \end{equation}
 If $\nu$ is large enough, then one can take $\eta=1/2,$
 which allows to get \eqref{eq:energy0} in the particular
  case $\bar\rho=\kappa=\mu=1.$  The general 
  case follows after scaling back in \eqref{eq:rescaling}.

\subsubsection*{Step 2: Exponential decay of a modified energy functional}

Observe that 
\begin{equation}\label{eq:equ} \d_t e+\div(eu)+P \div u= P(1)\div u,\end{equation}            
hence, integrating  on $\T^2,$ 
\begin{equation}\label{eq:reep}\frac{d}{dt}\int_{\T^2} e\dbar x=-\int_{\T^2} P\,\div u\dbar x.\end{equation}
Testing the momentum equation by $u,$ we get
\begin{equation}\label{eq:energystar}
\frac 12\frac{d}{dt} \int_{\T^2} \rho \abs{u}^2\dbar x+ \int_{\T^2} \abs{\nabla u}^2\dbar x+(\lambda+1)\int_{\T^2} \abs{\div u}^2\dbar x=-\int_{\T^2}\nabla P\cdot u\dbar x.\end{equation}

Now, performing an integration by parts in the last term, using \eqref{eq:gu} and putting  together with \eqref{eq:reep}  implies that
$$
\frac{d}{dt}D(t)+\int_{\T^2} \abs{\nabla \p u}^2\dbar x
+\nu\int_{\T^2} \abs{\div u}^2\dbar x=(1-\nu^{-1})\int_{\T^2}(P-1)\, \div u\dbar x,$$
with $$D(t):=\frac 12\int_{\T^2} \rho(t,x) \abs{u(t,x)}^2\dbar x+\frac{1}{\nu}\int_{\T^2} e(t,x)\dbar x.$$
Combining Cauchy-Schwarz inequality and \eqref{eq:energy0}  yields
$$\biggl|\int_{\T^2}(P-1)\, \div u\dbar x\biggr|\leq C_{\rho^*}\|a\|_{L_2}\|\div u\|_{L_2}
\leq \frac\nu4\|\div u\|_{L_2}^2+ \frac{C_{\rho^*}^2}{\nu}\|a\|_{L_2}^2.$$
Further remember that \eqref{eq:eqaimp} holds true.
Hence, setting 
$$\check D:=D-\frac1{2\nu}\int_{\T^2}(-\Delta)^{-1}\div(\rho u)\cdot a\dbar x,$$
we discover that
$$
\frac d{dt}\check D+\|\nabla\p u\|_{L_2}^2+\frac \nu4\|\div u\|_{L_2}^2
+\frac1{4\nu}\|a\|_{L_2}^2\leq \frac{C_0}\nu\|\nabla u\|_{L_2}^2+ \frac{C_{\rho^*}^2}{\nu}\|a\|_{L_2}^2.
$$
Arguing as in \eqref{eq:equaaru}, we see that $\check D\simeq D$ for large enough
$\nu,$ and Lemma \ref{lem:eseq} also ensures that the term $\|a\|_{L_2}^2$
may be replaced with $\|e\|_{L_1}$ in the left-hand side, 
up to a harmless constant depending only on $\rho^*$ and $\gamma.$
This allows to conclude that
$$
\frac d{dt}\check D+\frac18\|\nabla u\|_{L_2}^2+\frac{c_{\rho^*}}\nu\|e\|_{L_1}
\leq \frac{C_{\rho^*}^2}{\nu}\|a\|_{L_2}^2.
$$
Using \eqref{eq:poincareru} and remembering the 
definition of $D$ as well as the equivalence with $\check D$ , this may be rewritten
up to a change of $C_{\rho^*}$:
\begin{equation}\label{eq:dt}
    \frac d{dt}\check D+ \alpha\check D
\leq \frac{C_{\rho^*}^2}{\nu}\|e\|_{L_1}\ \hbox{ for all }\ 
\alpha\leq\min\biggl(c_{\rho^*},\frac1{8(\rho^*)^2}\biggr)\cdotp
\end{equation}
Hence, multiplying by $e^{\alpha t}$, then integrating and bounding 
$\|e(t)\|_{L_1}$ according to Inequality \eqref{eq:energy0}, we get
$$e^{\alpha t} \check D(t)\leq \check D(0) 
+2E_0\frac{C_{\rho^*}^2}{\nu}\int_0^t e^{(\alpha-\alpha_{0})\tau}\,d\tau.
$$
Choosing $\alpha=\frac12\min\bigl(c_{\rho^*},\frac1{8(\rho^*)^2}\bigr)$ and remembering 
 $\check D\simeq D$ completes the proof of \eqref{eq:energy1}. 

\subsubsection*{Step 3: Exponential decay of higher order norms}
In this step, we will prove  \eqref{eq:decayH1} 
and, defining $F_1$ according to \eqref{eq:F1}, 
\begin{multline}\label{eq:edl2h1}
\int^{\infty}_0\!\!\!\int_{\T^2} e^{\frac{\alpha_0}4 t}\biggl(\rho \abs{u_t}^2+ (h+P+\aP)\abs{\div u}^2\\
+\frac{(\bar P-1)^2}{\nu F_1^2(\rho^*)}+\frac{1}{\rho^*} \abs{\nabla^2 \p u}^2
+ \frac{1}{\rho^*} \abs{\nabla G}^2\biggr)\dbar x\,dt\leq C_{0,\rho^*}.
\end{multline}
The first thing to do is to test the momentum equation of (CNS) by $u_t.$ We get
\begin{multline}\label{eq:uh1es}
\frac{1}{2}\frac{d}{dt}\int_{\T^2} \abs{\nabla u}^2\dbar x+\frac{\lambda\!+\!1}{2}\frac{d}{dt}\int_{\T^2} \abs{\div u}^2\dbar x+\int_{\T^2} \rho \abs{u_t}^2\dbar x+\int_{\T^2} \nabla P \cdot u_t\dbar x\\=-\int_{\T^2} (\rho u\cdot \nabla u)\cdot u_t\dbar x.   
\end{multline}
To handle the pressure term,  we note that 
\begin{equation}\label{eq:eqp}
    P_t+\div(Pu)+h\div u=0 \with h:=\rho P'-P,
\end{equation}
 which, along with the definition of $G$ in \eqref{def:G}, gives
\begin{align}\label{eq:esp}
   \int_{\T^2} P_t \div u\dbar x&=\frac{1}{2\nu}\frac{d}{dt}\int_{\T^2} P^2\dbar x+\frac{1}{\nu}\int_{\T^2} P_t G\dbar x\nonumber\\
   &=\frac{1}{2\nu}\frac{d}{dt}\int_{\T^2} P^2\dbar x+\frac{1}{\nu}\int_{\T^2} Pu\cdot \nabla G\dbar x-\frac{1}{\nu}\int_{\T^2} h\,G\,\div u \dbar x.
\end{align}
Hence
$$\begin{aligned}
   \int_{\T^2} \nabla P\cdot u_t\dbar x&=-\int_{\T^2} P\cdot (\div u)_t\dbar x\\
   &=-\frac{d}{dt}\int_{\T^2} P\div u\dbar x+\int_{\T^2} P_t \div u\dbar x\\
   =-\frac{d}{dt}&\int_{\T^2} P\div u\dbar x+\frac{1}{2\nu}\frac{d}{dt}\int_{\T^2} P^2\dbar x+\frac{1}{\nu}\int_{\T^2} Pu\cdot \nabla G\dbar x-\frac{1}{\nu}\int_{\T^2} h\,G\,\div u\dbar x,
\end{aligned}$$
Plugging this identity in \eqref{eq:uh1es}, then using  \eqref{eq:gu} yields
$$\displaylines{
\quad
\frac{1}{2}\frac{d}{dt}\int_{\T^2} \abs{\nabla \p u}^2\dbar x+\frac{\nu}{2}\frac{d}{dt}\int_{\T^2} \abs{\div u}^2\dbar x
-\frac{d}{dt}\int_{\T^2} P\div u\dbar x+\frac{1}{2\nu}\frac{d}{dt}\int_{\T^2} P^2\dbar x\hfill\cr\hfill
+\int_{\T^2} \rho \abs{u_t}^2\dbar x
=-\int_{\T^2} (\rho u\cdot \nabla u)\cdot u_t\dbar x-\frac{1}{\nu}\int_{\T^2} Pu\cdot \nabla G\dbar x+\frac{1}{\nu}\int_{\T^2} h\,G\,\div u\dbar x
.}$$
Remembering  \eqref{def:G}, we get
$$\displaylines{
\quad
\frac{1}{2}\frac{d}{dt}\int_{\T^2} \abs{\nabla \p u}^2\dbar x
+\frac{1}{2\nu}\frac{d}{dt}\int_{\T^2} G^2\dbar x +\int_{\T^2} \rho \abs{u_t}^2\dbar x
\hfill\cr\hfill
=-\int_{\T^2} (\rho u\cdot \nabla u)\cdot u_t\dbar x-\frac{1}{\nu}\int_{\T^2} Pu\cdot \nabla G\dbar x+\frac{1}{\nu}\int_{\T^2} h\,G\,\div u\dbar x.}$$
Inserting \eqref{eq:eqedl2} divided by two into this inequality yields
\begin{multline}
\label{eq:new}
\frac{d}{dt}\biggl(\frac{1}{2}\int_{\T^2} \abs{\nabla \p u}^2\dbar x+\frac{1}{4}\int_{\T^2} \rho \abs{u}^2\dbar x+\frac{1}{2\nu}\int_{\T^2} G^2\dbar x+ \frac 12\int_{\T^2} e\dbar x\biggr)
\\
+\frac{\nu}{2}\int_{\T^2} \abs{\div u}^2\dbar x+\frac{1}{2}\int_{\T^2} \abs{\nabla \p u}^2\dbar x+\int_{\T^2} \rho \abs{u_t}^2\dbar x
\\
=-\int_{\T^2}(\rho u\cdot \nabla u)\cdot u_t\dbar x-\frac{1}{\nu}\int_{\T^2} Pu\cdot \nabla G\dbar x+\frac{1}{\nu}\int_{\T^2} h\,G\,\div u\dbar x.\end{multline}
In order to exhibit some exponential decay for $\|\tilde G\|_{L_2}$ and $\|\tilde P\|_{L_2}$,  we will use
repeatedly \eqref{eq:equ} and the fact that 
\begin{equation}\label{eq:divurprg}
\nu\div u=\tilde P+\tilde G.
\end{equation}
Now, integrating \eqref{eq:equ} on $\T^2$ yields
\begin{equation}\label{eq:eqtp}
    \frac 12\frac{d}{dt}\int_{\T^2} e\dbar x+\frac{1}{2\nu}\int_{\T^2} |\tilde P|^2\dbar x+\frac{1}{2\nu}\int_{\T^2} \tilde P \tilde G\dbar x=0,
\end{equation}
and 
\begin{equation}\label{eq:eqtg}
    \frac12\frac{d}{dt}\int_{\T^2} e\dbar x+\frac\nu2\int_{\T^2} \abs{\div u}^2\dbar x=\frac{1}{2\nu}\int_{\T^2} |\tilde G|^2\dbar x+\frac{1}{2\nu}\int_{\T^2} \tilde P \tilde G\dbar x.
\end{equation}
Hence, subtracting \eqref{eq:eqtg} to \eqref{eq:new},  we obtain
\begin{multline}\label{eq:eqfim1}
    \frac{d}{dt}\biggl(\frac{1}{2}\int_{\T^2} \abs{\nabla \p u}^2\dbar x+\frac{1}{4}\int_{\T^2} \rho \abs{u}^2\dbar x+\frac{1}{2\nu}\int_{\T^2} G^2\dbar x\biggr)+\int_{\T^2} \rho \abs{u_t}^2\dbar x\\
+\frac{1}{2\nu}\int_{\T^2} |\tilde G|^2\dbar x+\frac{1}{2}\int_{\T^2} \abs{\nabla \p u}^2\dbar x=-\int_{\T^2} (\rho u\cdot \nabla u)\cdot u_t\dbar x\\
-\frac{1}{2\nu}\int_{\T^2} \tilde P \tilde G\dbar x-\frac{1}{\nu}\int_{\T^2} Pu\cdot \nabla G\dbar x+\frac{1}{\nu}\int_{\T^2} hG\div u\dbar x.
\end{multline}
Next, testing \eqref{eq:eqp} with $\div u,$ using \eqref{eq:esp} and integrating by parts yields
\begin{equation}\label{eq:eqpim}
    \frac{1}{2\nu}\frac{d}{dt}\int_{\T^2} P^2\dbar x+\int_{\T^2} h\abs{\div u}^2\dbar x=\frac{1}{\nu}\int_{\T^2} hG\div u\dbar x-\frac{1}{2\nu}\int_{\T^2} P^2\div u\dbar x.
\end{equation}
As $\bar\rho=P(\bar\rho)=1,$ Lemma \ref{lem:eseq} implies that 
\begin{equation}\label{eq:eqtdp}
    \frac{1}{2\nu}\int_{\T^2} |\rP|^2\dbar x+\frac{(\aP-1)^2}{2\nu}\leq \frac{F_1(\rho^*)^2}{2\nu}\int_{\T^2} \abs{a}^2\dbar x.
\end{equation}
Putting the above three relations  together, we get
\begin{multline}\label{eq:1}
\frac12\frac{d}{dt}\biggl(\int_{\T^2} \abs{\nabla \p u}^2\dbar x+\frac{1}{2}\int_{\T^2} \rho \abs{u}^2\dbar x+\frac{1}{\nu}\int_{\T^2} (G^2+P^2)\dbar x\biggr)+\int_{\T^2} \rho \abs{u_t}^2\dbar x\\
  +\frac{1}{2\nu}\int_{\T^2} |\tilde G|^2\dbar x+\frac{1}{2}\int_{\T^2} \abs{\nabla \p u}^2\dbar x+\int_{\T^2} h\abs{\div u}^2\dbar x+\frac{1}{2\nu F_1^2(\rho^*)}\int_{\T^2} |\rP|^2\dbar x
  +\frac{(\aP-1)^2}{2\nu  F_1^2(\rho^*)}\\
\leq -\int_{\T^2} (\rho u\cdot \nabla u)\cdot u_t\dbar x-\frac{1}{2\nu}\int_{\T^2} \tilde P \tilde G\dbar x-\frac{1}{\nu}\int_{\T^2} Pu\cdot \nabla G\dbar x\\
+\frac{2}{\nu}\int_{\T^2} hG\div u\dbar x-\frac{1}{2\nu}\int_{\T^2} P^2 \div u\dbar x+\frac{1}{2\nu}\int_{\T^2}\abs{a}^2\dbar x.
\end{multline}
To handle the right-hand side, let us rewrite the momentum equation in terms of the viscous effective flux $G=\nu \div u-P $ as follows:
\begin{equation}\label{eq:ud2}
(\Delta u-\nabla \div u)+\nabla G=\rho \t{u}.
\end{equation}
From it, we discover that
\begin{equation}\label{eq:eqg2u}
\|\Delta \p u\|^2_{L_2}+\|\nabla G\|^2_{L_2}=\|\rho \t{u}\|^2_{L_2}\leq \rho^*\|\sqrt{\rho} \t{u}\|^2_{L_2}.
\end{equation}
In addition, we deduce from \eqref{eq:eqp}, \eqref{eq:divurprg} and $\aP=-\aG$ that
\begin{equation}\label{eq:rhs1}
\begin{aligned}
\frac{2}{\nu}\int_{\T^2} hG\,\div u \dbar x
&=-\frac{2}{\nu}\bar P\int_{\T^2} h\div u \dbar x+\frac{2}{\nu}\int_{\T^2} h\div u\;\tilde{G}\dbar x\\
&=\frac{2}{\nu}\bar P\int_{\T^2} P_t\dbar x+\frac{2}{\nu}\int_{\T^2} h\div u\;\tilde{G}\dbar x\\
&=\frac{1}{2\nu}\frac{d}{dt}{\bar P}^2+\frac{1}{2\nu}\frac{d}{dt}{\bar G}^2+\frac{2}{\nu}\int_{\T^2} h\div u\;\tilde{G}\dbar x,
\end{aligned}
\end{equation}
and
$$\begin{aligned}-\frac{1}{2\nu}\int_{\T^2} \div u\; P^2\dbar x
 &=-\frac{1}{2}\int_{\T^2} P\abs{\div u}^2\dbar x-\frac{\bar P}{2\nu}\int_{\T^2} \div u \;P\dbar x+\frac{1}{2\nu}\int_{\T^2} \div u\; P\tilde G\dbar x\\
 &=-\frac{1}{2}\int_{\T^2} P\abs{\div u}^2\dbar x-\frac{\aP}{2}\int_{\T^2} \abs{\div u}^2\dbar x+\frac{\aP}{2\nu}\int_{\T^2} \div u \; \tilde G\dbar x\\ &\hspace{7cm}+\frac{1}{2\nu}\int_{\T^2} \div u \;P\tilde G\dbar x.
\end{aligned}$$
Insert \eqref{eq:eqg2u}, \eqref{eq:rhs1} and the above inequality into \eqref{eq:1} and use 
\begin{equation}\label{eq:eqpgimp}
    \int_{\T^2} |G|^2\dbar x=|\aG|^2+\int_{\T^2} |\rG|^2\dbar x \andf \int_{\T^2} |P|^2\dbar x=|\aP|^2+\int_{\T^2} |\rP|^2\dbar x.
\end{equation}
Let
 $\displaystyle\Phi:=\frac{1}{2}\int_{\T^2} \abs{\nabla \p u}^2\dbar x+\frac{1}{2\nu}\int_{\T^2} \bigl(\rG^2+ \rP^2\bigr)\dbar x.$
{}From the previous relations,  we get
\begin{multline}\label{eq:22}
\frac{d}{dt}\biggl(\Phi+\frac{1}{4}\int_{\T^2} \rho \abs{u}^2\dbar x\biggr)+\frac{1}{2\nu}\int_{\T^2} |\tilde G|^2\dbar x+\frac{1}{2\nu F_1^2(\rho^*)}\int_{\T^2} |\rP|^2\dbar x\\+\frac{1}{2}\int_{\T^2} \abs{\nabla \p u}^2\dbar x
  +\frac{(\aP-1)^2}{2\nu  F_1^2(\rho^*)}+\int_{\T^2} \rho \abs{u_t}^2\dbar x+\int_{\T^2} h\abs{\div u}^2\dbar x
  +\int_{\T^2} \frac{(P+\aP)}{2}\abs{\div u}^2\dbar x\\
   +\frac{1}{4\rho^*}\int_{\T^2} \abs{\Delta \p u}^2\dbar x+\frac{1}{4\rho^*}\int_{\T^2} \abs{\nabla G}^2\dbar x\\   
\leq \frac 14\|\sqrt{\rho} \t{u}\|^2_{L_2}-\frac{1}{2\nu}\int_{\T^2} \tilde P \tilde G\dbar x-\int_{\T^2} \rho u\cdot \nabla u\cdot u_t\dbar x-\frac{1}{\nu}\int_{\T^2} Pu\cdot \nabla G\dbar x\\ +\frac{2}{\nu}\int_{\T^2} h\div u\, \rG\dbar x+\frac{\aP}{\nu}\int_{\T^2} \div u \,\tilde G\dbar x+\frac{1}{2\nu}\int_{\T^2} \div u \,P  \tilde G\dbar x+\frac{1}{2\nu}\int_{\T^2} \abs{a}^2\dbar x=:\sum_{j=1}^8 I_j.
\end{multline}
Now, from the definition of $\t u$ and Young's inequality, we have
\begin{equation}
I_1+I_3\leq \frac 34 \int_{\T^2} \rho \abs{u_t}^2\dbar x+\frac32\int_{\T^2} \rho \abs{u\cdot\nabla u}^2\dbar x.
\end{equation}
The first term can be absorbed  by the left-hand side of \eqref{eq:22}. For the second one, we use that
\begin{equation}\label{eq:denu}
u=\p u-\frac{1}{\nu}\nabla(-\Delta)^{-1}\rG-\frac{1}{\nu}\nabla(-\Delta)^{-1}\rP.
\end{equation} 
Hence, since $\nabla^2(-\Delta)^{-1}$ maps $L_4(\T^2)$ to itself and owing to \eqref{eq:GN} with $p=4,$ we get
\begin{equation*}
\begin{aligned}
\|\sqrt{\rho}u\!\cdot\! \nabla u\|^2_{L_2}&\leq 3\|\sqrt{\rho}u\!\cdot\! \nabla \p u\|^2_{L_2}\!+\!\frac{3}{\nu^2}\|\sqrt{\rho}u\!\cdot\! \nabla^2(-\Delta)^{-1}\rG\|^2_{L_2}\!+\!\frac{3}{\nu^2}\|\sqrt{\rho}u\!\cdot\! \nabla^2(-\Delta)^{-1}\rP\|^2_{L_2}\\
&\lesssim \sqrt{\rho^*}\|\rho^{\frac 14}u\|^2_{L_4}(\|\nabla \p u\|_{L_2}\|\nabla^2 \p u\|_{L_2}\!+\!\frac{1}{\nu^2}\|\rG\|_{L_2}\|\nabla G\|_{L_2}\!+\!\frac{1}{\nu^2}\|\rP\|_{L_2}\|\rP\|_{L_\infty})\\
&\leq  C\|\rho^{\frac 14}u\|^4_{L_4}\biggl({(\rho^*)^2}\|\nabla \p u\|^2_{L_2}+\frac{(\rho^*)^2}{\nu^4}\|\rG\|^2_{L_2}+\frac{\rho^*F^2_1(\rho^*)}{\nu^3}\|\rP\|^2_{L_\infty}\biggr) \\
&\hspace{2.8cm}+\frac{1}{12\rho^*}\|\nabla^2 \p u\|^2_{L_2}+\frac{1}{24\rho^*}\|\nabla G\|^2_{L_2}+\frac{1}{12\nu F^2_1(\rho^*)} \|\rP\|^2_{L_2}.
\end{aligned}
\end{equation*}
Remembering \eqref{eq:uL4}, 
the bound $\|\nabla u\|_{L_\infty(\R_+;L_2)}\leq C_0$ and  \eqref{eq:energy1}, we get 
\begin{equation}\label{eq:esucl4b}
\|\rho^{1/4} u\|_{L_4}^4\leq C_{0,\rho^*}\Bigl(e^{-\frac{\alpha_1}2t} 
+ \nu^{-1}e^{-\frac{\alpha_0}2t}\Bigr)\cdotp
\end{equation}
Hence, 
\begin{multline}\label{eq:eqnonlinear}
  \frac32 \|\sqrt{\rho}u\cdot \nabla u\|^2_{L_2}\leq C_{0,\rho^*}\Bigl(e^{-\frac{\alpha_1}2t} 
+ \nu^{-1}e^{-\frac{\alpha_0}2t}\Bigr)\\+\frac{1}{8\rho^*}\|\nabla^2 \p u\|^2_{L_2}+\frac{1}{16\rho^*}\|\nabla G\|^2_{L_2}+\frac{1}{8\nu F^2_1(\rho^*)} \|\rP\|^2_{L_2}.
\end{multline}
For $I_4$, we just write that
\begin{equation}
I_4\leq \frac{1}{\nu} (\rho^*)^{\gamma-\frac 12}\|\sqrt{\rho}u\|_{L_2}\|\nabla G\|_{L_2}\leq \frac{8(\rho^*)^{2\gamma}}{\nu^2}\|\sqrt{\rho}u\|^2_{L_2}+\frac{1}{32 \rho^*}\|\nabla G\|^2_{L_2}.
\end{equation}
For the other terms, one takes advantage of Young's inequality and \eqref{eq:poinca} to get 
$$
\begin{aligned}
&I_2=-\frac{1}{2\nu}\int_{\T^2} \tilde P \tilde G\dbar x\leq  \frac{1}{8\nu F^2_1(\rho^*)}\int_{\T^2} |\rP|^2 \dbar x+\frac{c_{\T^2}^2F^2_1(\rho^*)}{2\nu} \int_{\T^2} \abs{\nabla G}^2 \dbar x,\\
&I_5=\frac{2}{\nu}\int_{\T^2} h\div u\;\rG\dbar x\leq \frac 12 \int_{\T^2} h\abs{\div u}^2\dbar x+\frac{2\|h\|_{L_\infty}}{\nu^2}\int_{\T^2} |\rG|^2\dbar x,\\
&I_6=\frac{\aP}{2\nu}\int_{\T^2} \div u \; \tilde G\dbar x\leq \frac{\aP}{4}\int_{\T^2} \abs{\div u}^2 \dbar x+\frac{\aP}{4\nu^2}\int_{\T^2} |\rG|^2\dbar x,\\
&I_7=\frac{1}{2\nu}\int_{\T^2} \div u \;P \tilde G\dbar x\leq \frac 14\int_{\T^2} P\abs{\div u}^2\dbar x+\frac{\|P\|_{L_\infty}}{4\nu^2}\int_{\T^2} |\rG|^2\dbar x.
\end{aligned}$$
Let us assume that
$$\nu\geq\max\bigl(16\rho^*c_{\T^2}^2 F^2_1(\rho^*), 8\|h\|_{L_\infty}+\bar P+\|P\|_{L_\infty},4(\rho^*)^{\gamma}\sqrt{\alpha_1}\bigr)$$
so that in particular  the sum of the coefficients of the last term in $I_5,I_6,I_7$ is smaller than $\frac{1}{4\nu}$ and the coefficients of the first term in $I_4$ is smaller than $\alpha_1/2$
(where $\alpha_1$ has been defined in Inequality \eqref{eq:energy1}). 
 Then, reverting to \eqref{eq:22}, we end up with
\begin{equation*}
\displaylines{
\frac{d}{dt}\biggl(\Phi+\frac{1}{4}\int_{\T^2} \rho \abs{u}^2\dbar x\biggr)+\frac{\Phi}{4  F_1^2(\rho^*)}+\frac 12\int_{\T^2} \rho \abs{u_t}^2\dbar x+\frac{(\aP-1)^2}{2\nu  F_1^2(\rho^*)}\hfill\cr\hfill
 +\frac 12\int_{\T^2} h\abs{\div u}^2\dbar x+\int_{\T^2} \frac{(P+\aP)}{4}\abs{\div u}^2\dbar x+\frac{1}{8\rho^*}\int_{\T^2} \abs{\nabla^2 \p u}^2\dbar x\hfill\cr\hfill+\frac{1}{16\rho^*}\int_{\T^2} \abs{\nabla G}^2\dbar x\leq  C_{0,\rho^*}\Bigl(e^{-\frac{\alpha_1}2t} + \nu^{-1}e^{-\frac{\alpha_0}2t}\Bigr)+\frac{1}{2\nu c_\gamma}\int_{\T^2} e\dbar x.}\end{equation*}
Inserting half of \eqref{eq:dt} (with $\alpha=\alpha_0/8)$ into the above inequality implies that
$$\displaylines{\frac{d}{dt}\Psi+\frac{\Phi}{4  F_1^2(\rho^*)}+\frac {\alpha_1}2 \check D_1+\frac 12\int_{\T^2} \rho \abs{u_t}^2\dbar x+\frac{(\aP-1)^2}{2\nu  F_1^2(\rho^*)}\hfill\cr\hfill
   +\frac 12\int_{\T^2} h\abs{\div u}^2\dbar x+\int_{\T^2} \frac{(P+\aP)}{4}\abs{\div u}^2\dbar x+\frac{1}{8\rho^*}\int_{\T^2} \abs{\nabla^2 \p u}^2\dbar x\hfill\cr\hfill
+\frac{1}{16\rho^*}\int_{\T^2} \abs{\nabla G}^2\dbar x\leq 
 C_{0,\rho^*}\Bigl(e^{-\frac{\alpha_1}2t} + \nu^{-1}e^{-\frac{\alpha_0}2t}\Bigr)
+\frac{C_{\rho^*,\gamma}}{\nu}\int_{\T^2} e\dbar x,}$$
with 
$$\check  D_1:=\frac 12\int_{\T^2} \rho \abs{u}^2\dbar x+\frac {1}{2\nu}\int_{\T^2} e\dbar x
-\frac{1}{4\nu}\int_{\T^2}(-\Delta)^{-1}\div(\rho u)\cdot a\dbar x \andf \Psi:= \Phi + \check D_1.
 $$
Then, bounding 
$\|e\|_{L_1}$ in the above inequality  according to Inequality \eqref{eq:energy1}
and taking $\wt \alpha\leq\min (\alpha_1, 1/(4F_1^2(\rho^*)))$ yields  
$$
\displaylines{
\frac{d}{dt}\Psi+\wt \alpha \Psi+\frac 12\int_{\T^2} \rho \abs{u_t}^2\dbar x
   +\frac 12\int_{\T^2} h\abs{\div u}^2\dbar x+\int_{\T^2} \frac{(P+\aP)}{4}\abs{\div u}^2\dbar x\hfill\cr\hfill
   +\frac{(\aP-1)^2}{2\nu  F_1^2(\rho^*)}+\frac{1}{8\rho^*}\int_{\T^2} \abs{\nabla^2 \p u}^2\dbar x+\frac{1}{16\rho^*}\int_{\T^2} \abs{\nabla G}^2\dbar x
\leq  C_{0,\rho^*}\Bigl(e^{-\frac{\alpha_1}2t} + \nu^{-1}e^{-\frac{\alpha_0}2t}\Bigr)\cdotp}
$$
 Taking $\wt\alpha$ such that $\alpha_0<2\wt\alpha<\alpha_1$  (which
 is always possible if $\nu$ is large enough),   multiplying by $e^{\wt\alpha t}$  then integrating yields 
 \begin{multline}\label{eq:Psi}e^{\wt\alpha t}\Psi(t) +\int_0^te^{\wt \alpha \tau}\!\!
\int_{\T^2}\biggl(\frac{\rho \abs{u_\tau}^2}2
   +\frac{h\abs{\div u}^2}2+\frac{(P+\aP)}{4}\abs{\div u}^2
   +\frac{(\aP-1)^2}{2\nu  F_1^2(\rho^*)}\hfill\cr\hfill+\frac{1}{8\rho^*} \abs{\nabla^2 \p u}^2+\frac{\abs{\nabla G}^2}{16\rho^*}\biggr)\!\dbar x\,d\tau
\leq \Psi(0) + \frac{C_{0,\rho^*}}{\alpha_1-2\wt\alpha}
+\frac{C_{0,\rho^*}}{\nu(2\wt\alpha-\alpha_0)}
e^{(\wt\alpha-\frac{\alpha_0}2)t}\end{multline}
 which, owing to 
$\check D_1\simeq D$ gives \eqref{eq:decayH1}. 
\smallbreak
In order to get \eqref{eq:edl2h1}, we observe that \eqref{eq:Psi} implies that
\begin{multline}\label{eq:ftau}\int_0^te^{\wt \alpha \tau} f(\tau)\,d\tau\leq 
C_{0,\rho^*}\bigl(1 +\nu^{-1} e^{(\wt \alpha-\frac{\alpha_0}2)t}\bigr)\quad\hbox{for all }
t\geq0,\with\\
f:=\frac{(\aP-1)^2}{2\nu  F_1^2(\rho^*)}+
\int_{\T^2}\biggl(\frac{\rho \abs{u_t}^2}2
   +\frac{h\abs{\div u}^2}2+\frac{(P+\aP)}{4}\abs{\div u}^2+\frac{1}{8\rho^*} \abs{\nabla^2 \p u}^2+\frac{\abs{\nabla G}^2}{8\rho^*}\biggr)\!\dbar x.\end{multline}
Now, integrating by parts, we see that 
$$\begin{aligned}\int_0^t e^{\frac{\alpha_0}4\tau} f(\tau)\,d\tau
&=\int_0^t e^{(\frac{\alpha_0}4-\wt\alpha)\tau} \bigl(e^{\wt \alpha \tau}f(\tau)\bigr)d\tau\\
&=e^{(\frac{\alpha_0}4-\wt\alpha)t}\int_0^te^{\wt \alpha \tau}f(\tau)\,d\tau
+\Bigl(\wt\alpha-\frac{\alpha_0}4\Bigr)\int_0^t e^{(\frac{\alpha_0}4-\wt\alpha)\tau}
\biggl(\int_0^\tau e^{\wt \alpha \tau'}f(\tau')\,d\tau'\biggr)d\tau.
\end{aligned}$$
Then, using Inequality \eqref{eq:ftau} to bound the right-hand gives 
\eqref{eq:edl2h1}.

\subsubsection*{Step 4: Exponential decay of the pressure}

Let  $\overline{h \div u}$ be  the mean value of $h\div u.$
    From   \eqref{eq:eqp} we gather that $\rP$ satisfies
$$(\rP)_t+\div(\rP u)+{(h+\bar P)\div u}=\overline{h \div u}.$$
Hence, multiplying the above equation by $e^{2\beta t}\wt P$ and integrating on $\T^2$, one gets
$$ \frac 12\frac{d}{dt}e^{2\beta t}\int_{\T^2}\rP^2 \dbar x=\beta \int_{\T^2}e^{2\beta t}\rP^2\dbar x-\frac 12\int_{\T^2}e^{2\beta t} \div u\, \rP^2\dbar x-\int_{\T^2}e^{2\beta t}(h+\aP)\div u\cdot \rP\dbar x.$$
We have then by \eqref{eq:decayH1} and H\"older inequality:
$$\begin{aligned}
    \frac 12\frac{d}{dt}e^{2\beta t}\int_{\T^2}|\rP|^2 \dbar x\leq& 
    \beta \nu e^{2\beta t} C_{0,\rho^*} \Bigl(e^{-\frac{\alpha^*}2 t}+\nu^{-1}e^{-\frac{\alpha_0}2t}\Bigr)\\
    &+\bigl(\nu^{\frac 12}\|\div u\|_{L_2}\bigr)
    \bigl(\nu^{-\frac 12}\|\rP\|_{L_2}\bigr)e^{2\beta t}\Big\|h+\frac12(\aP+P)\Big\|_{L_\infty}\\
    \leq& e^{2\beta t} C_{0,\rho^*} \Bigl(e^{-\frac{\alpha^*}2 t}+\nu^{-1}e^{-\frac{\alpha_0}2t}\Bigr)\Bigl( \beta \nu+
    \Big\|h+\frac12(\aP+P)\Big\|_{L_\infty}\Bigr)\cdotp
\end{aligned}$$
Now,   integrating with respect to $t$ 
and taking   $\beta=\alpha_0/4$ gives \eqref{eq:decayP}.

\subsubsection*{Step 5:  Upper bound of the density}

This has been proved in \cite{DM-ripped}, and this does not require exponential decay estimates. 

Reverting to the previous steps, this  completes the proof of \eqref{eq:energy0}, \eqref{eq:energy1}, 
\eqref{eq:decayH1}, \eqref{eq:decayL2t}  and \eqref{eq:decayP}.

\subsubsection*{Step 6: Lower bound of the density}  

We leave aside the case $\rho_{0,*}=0$ and assume that 
\begin{equation}\label{eq:infrho0}\rho_*:=\underset{[0,T]\times \T^2}{\sup} \geq \rho_{0,*}/2>0.\end{equation}
We claim that if  $\nu$ is large enough (independently of $T$), then  \eqref{eq:infrho0}  actually holds true \emph{with a strict inequality}.
To prove it, we observe that 
$$\d_t\log \rho+u\cdot \nabla \log \rho=-\frac{1}{\nu}(\rP+\rG),$$
and that 
applying $\div$ to the  momentum equation of (CNS) gives 
$$\Delta \rG=\d_t(\div (\rho u))+\div(\div(\rho u\otimes u)).$$
Hence, following \cite{BD1997} and introducing 
$$\varphi:=\log \rho-\nu^{-1}(-\Delta)^{-1}\div(\rho u),$$
we discover that
\begin{equation}\label{eq:vpr}
    \d_t \varphi+u\cdot \nabla \varphi+\frac{\rP}{\nu}=-\frac{1}{\nu}\sum_{i,j} [u^j,(-\Delta)^{-1}\d_i\d_j]\rho u^i.\end{equation}
Observe that the map $x\mapsto \frac{e^x-1}{x}$ is increasing on $\R.$
Hence, taking $x=\log \rho^{\gamma}$ yields
$$\begin{aligned}
    \rho^{\gamma}&\leq 1+M_-\log \rho\with M_-:=\frac{\rho_*^{\gamma}-1}{\log \rho_*}&\quad\hbox{if }\  \rho_*\leq \rho<1,\\
\rho^{\gamma}&\leq 1+ M_+\log \rho\with M_+:=\frac{(\rho^*)^{\gamma}-1}{\log \rho^*}&\quad\hbox{if }\
 1\leq \rho\leq \rho^*.   
\end{aligned}$$
Let  $M(\rho)=M_-$ if $\rho_*\leq \rho<1$ and $M(\rho)=M_+$ if $1\leq \rho\leq \rho^*.$
{}From  \eqref{eq:vpr}, we have 
\begin{equation*}
    \begin{aligned}
        \d_t (-\varphi)+u\cdot \nabla (-\varphi)-\frac{M\varphi}{\nu}
        &=\frac{P-M\varphi}{\nu}-\frac{\aP}{\nu}+\frac{1}{\nu}[u^j,(-\Delta)^{-1}\d_i\d_j]\rho u^i\\
        &\leq \frac{1-\aP}{\nu}+\frac{1}{\nu} [u^j,(-\Delta)^{-1}\d_i\d_j]\rho u^i+\frac{M}{\nu^2}(-\Delta)^{-1}\div(\rho u)\cdotp
    \end{aligned}
\end{equation*}
Hence, setting $\varphi^{-}=\max\{-\varphi,0\},$ $M_*=\min(M_-,M_+)$ and $M^*=\max(M_-,M_+)$ yields
$$\d_t \varphi^{-}+u\cdot\nabla \varphi^{-}+\frac{M_*}{\nu}\varphi^{-}+\frac{\aP-1}{\nu}\leq \frac{1}{\nu}\abs{[u^j,(-\Delta)^{-1}\d_i\d_j]\rho u^i}+\frac{M^*}{\nu^2}\abs{(-\Delta)^{-1}\div(\rho u)}\cdotp$$
As $\aP-1=(\gamma-1)\int_{\T^2} e\dbar x\geq 0,$  this gives
$$\d_t \varphi^{-}+u\cdot\nabla \varphi^{-}+\frac{M_*}{\nu}\varphi^{-}\leq \frac{1}{\nu}\abs{[u^j,(-\Delta)^{-1}\d_i\d_j]\rho u^i}+\frac{M^*}{\nu^2}\abs{(-\Delta)^{-1}\div(\rho u)},$$
whence
\begin{multline}\label{eq:f-1}
\|\varphi^{-}(t)\|_{L_\infty}\leq e^{-\frac{M_*}{\nu}t}\|\varphi^{-}(0)\|_{L_\infty}+ \frac{1}{\nu}\int_0^t e^{-\frac{M_*}{\nu}(t-\tau)}\|[u^j,(-\Delta)^{-1}\d_i\d_j]\rho u^i\|_{L_\infty}\,d\tau\\
+\frac{M^*}{\nu^2}\int_0^t e^{-\frac{M_*}{\nu}(t-\tau)}\|(-\Delta)^{-1}\div(\rho u)\|_{L_\infty}\,d\tau.
\end{multline}
To bound the integrals, we proceed as in \cite{Des-CPDE} :  on the one hand 
Sobolev embedding and the continuity of Riesz operator ensure that
\begin{equation}\label{eq:esof2rhs}
\|(-\Delta)^{-1}\div(\rho u)\|_{L_\infty}\lesssim\|\nabla (-\Delta)^{-1}\div(\rho u)\|_{L_4}\lesssim \|\rho u\|_{L_4}.
\end{equation} On the other hand,  Sobolev embedding and  Coifman, Lions, Meyer and Semmes inequality \cite{RPYS} enable us to write that
\begin{align}\|[u^j,(-\Delta)^{-1}\d_i\d_j]\rho u^i\|_{L_\infty}&\lesssim \|\nabla u\|_{L_{12}}\|\rho u\|_{L_4}\nonumber\\\label{eq:CMLS}
&\lesssim \left(\|\nabla ^2 \p u\|_{L_2}+\nu^{-1}(\|\nabla G\|_{L_2}+\|\rP\|_{L_\infty})\right)\|\rho u\|_{L_4} \cdotp
\end{align}
Plugging \eqref{eq:esof2rhs} and \eqref{eq:CMLS}  in  \eqref{eq:f-1}, then using \eqref{eq:edl2h1}, \eqref{eq:esucl4b} we end up with
$$\|\varphi^{-}(t)\|_{L_\infty}\!\leq\! \|\varphi^{-}(0)\|_{L_\infty}+ C_{0,\rho^*}\biggl(\frac1{\nu^{3/4}}
+\frac{M^*}{\nu M_*}\biggr)\cdotp$$
Hence, using the definition of $\varphi$  and, again, \eqref{eq:esucl4b}, \eqref{eq:esof2rhs}, we get
$$\log\Bigl(\frac{1}{\rho}\Bigr)\leq   \log\Bigl(\frac{1}{\rho_{0,*}}\Bigr)+  C_{0,\rho^*}\biggl(\frac1{\nu^{3/4}}
+\frac{M^*}{\nu M_*}\biggr)
   \quad\hbox{on }\ [0,T].    $$
In other words,
$$\frac{\rho(t)}{\rho_{0,*}}\geq\exp\biggl(-\frac{C_{0,\rho^*}M^*}{\nu M_*}\biggr)
\exp\biggl(-\frac{C_{0,\rho^*}}{\nu^{3/4}}\biggr)\quad\hbox{for all }\ t\in[0,T].$$
It is clear that if $\nu$ is large enough, then the right-hand side is strictly larger than $1/2,$ and combining with a bootstrap argument allows  to  complete the proof.\qed


\subsection{Time weighted estimates with exponential decay}

In this section, we prove Theorem \ref{them:edcns2}.
We need first the following intermediate result. 
\begin{proposition}\label{p:timedecay1} Under the assumptions and notations of Theorem \ref{themc1} in the particular case
$\mu=\bar\rho=P'(\bar\rho)=1,$  we have
  \begin{align}\label{eq:tweug1}
    t\|\nabla \p u(t)\|^2_{L_2}+\nu^{-1}t\|\rG(t)\|^2_{L_2}&\leq C_{0}\bigl(\nu^{-1}e^{-\frac{\alpha_0}3t}
+e^{-\alpha_2 t}\bigr),\\
         \label{eq:twel2l2}
     \int_0^\infty e^{2\delta t}\bigl(t\|\sqrt\rho\, \t u\|_{L_2}^2+\|\nabla \p u\|_{L_2}^2&+\nu\|\div u\|_{L_2}^2+t\|\nabla G\|^2_{L_2}+t\|\nabla^2\p u\|_{L_2}^2\bigr)\,dt
     \leq C_{0}
     \end{align}
with $\alpha_2:=\min(\alpha^*,c^{-2}_{\T^2}(\rho^*)^{-3})/4$ and $\delta:=\min(\alpha_2,\alpha_0)/4.$
\end{proposition}
\begin{proof}
Testing the momentum equation of  $(CNS)$ with $t\t u$ gives
\begin{multline*}
    \frac{1}{2}\frac{d}{dt}\int_{\T^2}t\abs{\nabla u}^2\dbar x+\frac{\lambda+1}{2}\frac{d}{dt}\int_{\T^2}t\abs{\div u}^2\dbar x+\int_{\T^2}t\rho \abs{\t u}^2\dbar x\\
    =-\int_{\T^2}\nabla P\cdot tu_t \dbar x+\int_{\T^2}t\rho \t u\cdot (u\cdot \nabla u)\dbar x+ \frac{1}{2}\int_{\T^2}\abs{\nabla u}^2\dbar x+\frac{\lambda+1}{2}\int_{\T^2}\abs{\div u}^2\dbar x.
\end{multline*}
To handle the pressure term, we use   \eqref{eq:eqp} and integrate by parts to get 
$$\displaylines{\quad-\int_{\T^2}\nabla P\cdot tu_t \dbar x=\frac{d}{dt}\int_{\T^2}P t\div u\dbar x+\int_{\T^2}\div(P u)\cdot t\div u\dbar x\hfill\cr\hfill
+\int_{\T^2}h\div u\cdot t\div u\dbar x-\int_{\T^2}P \div u\dbar x.\quad}$$
Putting the above two equations and \eqref{eq:energystar}, \eqref{eq:gu} together, we obtain
\begin{multline*}
    \frac{1}{2}\frac{d}{dt}\int_{\T^2}t\abs{\nabla \p u}^2\dbar x+\frac{\nu}{2}\frac{d}{dt}\int_{\T^2}t\abs{\div u}^2\dbar x+ \frac{1}{2}\frac{d}{dt}\int_{\T^2}\rho \abs{u}^2\dbar x-\frac{d}{dt}\int_{\T^2}P t\div u\dbar x\\
    +\int_{\T^2}t\rho \abs{\t u}^2\dbar x+ \frac{1}{2}\int_{\T^2}\abs{\nabla \p u}^2\dbar x+\frac{\nu}{2}\int_{\T^2}\abs{\div u}^2\dbar x\\
   = \int_{\T^2}t\rho \t u\cdot (u\cdot \nabla u)\dbar x+\int_{\T^2}\div(P u)\cdot t\div u\dbar x+\int_{\T^2}th(\div u)^2\,\dbar x,
\end{multline*}
which along with the fact that 
$$\frac\nu2(\div u)^2=\frac1{2\nu}(G^2-P^2)+P\div u$$
and the following consequence of  \eqref{eq:eqpim} and \eqref{eq:rhs1},
$$\begin{aligned}
    \frac{1}{2\nu}\frac{d}{dt}\int_{\T^2}  t P^2\dbar x&+\int_{\T^2} t h\abs{\div u}^2\dbar x\\&=\frac{1}{\nu}\int_{\T^2}  t G\cdot h \div u\dbar x-\frac{1}{2\nu}\int_{\T^2} t P^2\div u\dbar x+\frac{1}{2\nu}\int_{\T^2}P^2\dbar x
    \\
&= \frac{1}{2\nu}\frac{d}{dt}(t\abs{\aG}^2)-\frac{1}{2\nu}\abs{\aG}^2+\frac{t}{\nu} \int_{\T^2}h\div u \cdot \rG\dbar x\\
& \hspace{3cm}-\frac{1}{2\nu}\int_{\T^2} t P^2\div u\dbar x+\frac{1}{2\nu}\int_{\T^2}P^2\dbar x. 
       \end{aligned} $$
Remembering $\aP=-\aG$ and using \eqref{eq:eqpgimp}, the above three equalities gives
\begin{multline}\label{eq:poidstd}
    \frac{1}{2}\frac{d}{dt}\int_{\T^2}t\abs{\nabla \p u}^2\dbar x+\frac{1}{2\nu}\frac{d}{dt}\int_{\T^2}t\abs{\rG}^2\dbar x+ \frac{1}{2}\frac{d}{dt}\int_{\T^2}\rho \abs{u}^2\dbar x\\
    +\int_{\T^2}t\rho \abs{\t u}^2\dbar x+ \frac{1}{2}\int_{\T^2}\abs{\nabla \p u}^2\dbar x+\frac{\nu}{2}\int_{\T^2}\abs{\div u}^2\dbar x
   = \int_{\T^2}t\rho \t u\cdot (u\cdot\nabla u)\dbar x \\+\int_{\T^2}\div(P u)\cdot t\div u\dbar x
   +\frac{1}{2\nu}\int_{\T^2}\rP^2\,dx-\frac{t}{2\nu}\int_{\T^2}P^2\div u\,dx+\frac{t}{\nu} \int_{\T^2}h\div u \cdot \rG\dbar x.
\end{multline}
Finally, taking advantage of \eqref{eq:eqg2u}
to incorporate $t\|\nabla^2\p u\|_{L_2}^2$ and
$t\|\nabla G\|_{L_2}^2$ in the left-hand side of \eqref{eq:poidstd}
and denoting 
$$\Pi(t):=\int_{\T^2}\bigl(t(\abs{\nabla \p u(t)}^2+
\nu^{-1}|\rG(t)|^2)+\rho(t)|u(t)|^2\bigr)\dbar x, $$
we end up with 
\begin{multline}\label{eq:Pi}
\frac d{dt}\Pi + 
\int_{\T^2}t\bigl(\rho|\t u|^2+\frac1{\rho^{*}}(|\nabla^2\p u|^2+|\nabla G|^2)\bigr)\dbar x
+\int_{\T^2}\bigl(\abs{\nabla \p u}^2+\nu\abs{\div u}^2\bigr)
\dbar x\\\leq 
\frac{1}{\nu}\int_{\T^2}\rP^2\,dx
+ 2\int_{\T^2}t\rho \t u\cdot (u\cdot\nabla u)\dbar x +2\int_{\T^2}\div(P u)\cdot t\div u\dbar x\\
   -\frac{1}{\nu}\int_{\T^2}tP^2\div u\,dx+\frac{2}{\nu} \int_{\T^2}th\div u \cdot \rG\dbar x.
\end{multline}
The first term in the right-hand side can 
be bounded according to \eqref{eq:decayP}. 
For the second one,
 we argue as for proving \eqref{eq:eqnonlinear}
except that we use
$$
C\frac1{\nu^2}\|\wt P\|_{L_2}\|\wt P\|_{L_\infty}\leq \frac1{4\nu^2}\|\wt P\|_{L_2}^2+\frac{C^2}{\nu^2}
\|\wt P\|_{L_\infty}.$$
so that we get $$\displaylines{
 2\int_{\T^2}t\rho \t u\cdot (u\cdot \nabla u)\dbar x\leq \frac{t}{2}\|\sqrt{\rho} \t u\|^2_{L_2}
+\frac{t}{4\rho^*}\|\nabla^2 \p u\|^2_{L_2}
+\frac{t}{8\rho^*}\|\nabla G\|^2_{L_2}+\frac{t}{2\nu^2}\|\rP\|^2_{L_2}
  \hfill\cr\hfill
 +C_{0}t\Bigl(e^{-\frac{\alpha^*}2t} 
+ \nu^{-1}e^{-\frac{\alpha_0}2t}\Bigr)\biggl({(\rho^*)^2}\|\nabla \p u\|^2_{L_2}+\frac{(\rho^*)^2}{\nu^4}\|\rG\|^2_{L_2}+\frac{\rho^*}{\nu^2}\|\rP\|^2_{L_\infty}\biggr)\cdotp}$$
Hence, bounding the terms with $\nabla\p u$ and $\rG$ according to \eqref{eq:decayH1},  we conclude that
\begin{multline*}
2\int_{\T^2}t\rho \t u\cdot (u\cdot \nabla u)\dbar x\leq \frac{t}{2}
\|\sqrt{\rho} \t u\|^2_{L_2}
+\frac{t}{4\rho^*}\|\nabla^2 \p u\|^2_{L_2}
+\frac{t}{8\rho^*}\|\nabla G\|^2_{L_2}+\frac{t}{2\nu^2}\|\rP\|^2_{L_2}\\
+C_{0}t\Bigl(\Bigl(e^{-\frac{\alpha^*}2t} 
+ \nu^{-1}e^{-\frac{\alpha_0}2t}\Bigr)^2 +
\nu^{-2}\Bigl(e^{-\frac{\alpha^*}2t} 
+ \nu^{-1}e^{-\frac{\alpha_0}2t}\Bigr)\Bigr)\cdotp
\end{multline*}
Next, using \eqref{eq:decayH1},  we have 
 $$\frac{t}{\nu} \int_{\T^2}\rP^2 \div u\dbar x\leq \frac{t}{\nu} \|\div u\|_{L_2}\|\rP\|_{L_\infty}\|\rP\|_{L_2}
 \leq \frac{C_{0}}{\nu}  t\Bigl(e^{-\frac{\alpha^*}2t} 
+ \nu^{-1}e^{-\frac{\alpha_0}2t}\Bigr)\cdotp $$
Hence,  observing that
 $$ \frac{t}{\nu} \int_{\T^2} P^2 \div u\dbar x
\leq\frac{t}{\nu} \int_{\T^2}\rP^2 \div u\dbar x+\frac {2t}\nu
|\aP| \|\rP\|_{L_2}\|\div u\|_{L_2}$$
and arguing as above for bounding the last term,  we discover that
\begin{equation*}
\frac{t}{\nu} \int_{\T^2}P^2 \div u\dbar x\leq  
  \frac{C_{0}}{\nu}  t\Bigl(e^{-\frac{\alpha^*}2t} 
+ \nu^{-1}e^{-\frac{\alpha_0}2t}\Bigr)\cdotp   
\end{equation*}
To handle the next term, we use \eqref{eq:divurprg} then integrate by parts to get
$$    t\int_{\T^2}\div(Pu)\cdot \div u\dbar x
=\frac{t}{2\nu}  \int_{\T^2}\div u\cdot P^2\dbar x-\frac{t}{\nu}\int_{\T^2} P\cdot u\cdot \nabla G\dbar x.$$
The first term may be treated as above and for the second one, 
we use that  thanks to H\"older and Young inequalities,  \eqref{eq:pti}, \eqref{eq:decayH1}, then  \eqref{eq:eqg2u},
$$\begin{aligned}
-\frac{t}{\nu}\int_{\T^2} P\cdot u\cdot \nabla G\dbar x
&\leq \frac{4t\rho^*}{\nu^2} \|P\|^2_{L_\infty}\| u\|^2_{L_2}+\frac{t}{16\rho^*}\|\nabla G\|^2_{L_2}\\
    &\leq \frac{4c^2_{\T^2}t(\rho^*)^3}{\nu^2} \|P\|^2_{L_\infty}\|\nabla u\|^2_{L_2} +\frac{t}{16\rho^*}\|\nabla G\|^2_{L_2}\\
 &\leq  C_{0}\nu^{-2}t\Bigl(e^{-\frac{\alpha^*}2t} 
+ \nu^{-1}e^{-\frac{\alpha_0}2t}\Bigr) +\frac{t}{16\rho^*}\|\nabla G\|^2_{L_2}.\end{aligned}$$
To handle the last term of \eqref{eq:Pi}, using  again  Young inequality and \eqref{eq:decayH1} gives
$$\begin{aligned}
    \frac{t}{\nu} \int_{\T^2}h\div u \cdot \rG\leq& \frac{t}{16\rho^*}\|\nabla G\|^2_{L_2}+ c_{\T^2}^4\frac{t}{\nu^2} \|\div u\|^2_{L_2}\\
    \leq & \frac{t}{16\rho^*}\|\nabla G\|^2_{L_2}
    +C_{0}\nu^{-3}t\Bigl(e^{-\frac{\alpha^*}2t} 
+ \nu^{-1}e^{-\frac{\alpha_0}2t}\Bigr)\cdotp
\end{aligned}$$
Plugging all the above estimates in \eqref{eq:Pi}, we get 
$$\displaylines{
\frac d{dt}\Pi + \frac12\int_{\T^2}t\bigl(\rho|\t u|^2+\frac1{\rho^{*}}(|\nabla^2\p u|^2+|\nabla G|^2)\bigr)\dbar x
+\int_{\T^2}\bigl(\abs{\nabla \p u}^2+\nu\abs{\div u}^2\bigr)
\dbar x\hfill\cr\hfill\leq 
C_{0}\Bigl(\nu^{-1}e^{-\frac{\alpha_0}2t}
+ \nu^{-1}t\Bigl(e^{-\frac{\alpha^*}2t} 
+ \nu^{-1}e^{-\frac{\alpha_0}2t}\Bigr)
+ t\Bigl(e^{-\frac{\alpha^*}2t} 
+ \nu^{-1}e^{-\frac{\alpha_0}2t}\Bigr)^2
\Bigr)\cdotp}$$
At this stage, the fundamental observation is that
the basic Poincar\'e inequality \eqref{eq:poinca} and \eqref{eq:poincareru} imply that 
$$\frac14\int_{\T^2}t\bigl(\rho|\t u|^2+\frac1{\rho^{*}}(|\nabla^2\p u|^2+|\nabla G|^2)
+\abs{\nabla u}^2\bigr)\dbar x\geq \alpha_2 \Pi
\with \alpha_2 := \frac1{4c_{\T^2}^2(\rho^*)^3}\cdotp$$
Hence after keeping only 
the main order terms (for large $\nu$) 
in the right-hand side, the above inequality implies that
$$\displaylines{\frac d{dt}\Pi+\alpha_2\Pi + \Xi
\leq C_{0}\Bigl(\nu^{-1}e^{-\frac{\alpha_0}2t}
+te^{-\frac{\alpha^*}2t} 
+ \nu^{-2}te^{-\frac{\alpha_0}2t}\Bigr)\hfill\cr\hfill\with 
\Xi(t):=  \frac14\int_{\T^2}t\bigl(\rho|\t u|^2+\frac1{\rho^{*}}(|\nabla^2\p u|^2+|\nabla G|^2)\bigr)\dbar x
+\frac12\int_{\T^2}\bigl(\abs{\nabla \p u}^2+\nu\abs{\div u}^2\bigr)\dbar x.}$$
Then, integrating yields assuming that $\nu$ is large 
enough and that $2\alpha_2>\alpha_0$
$$\displaylines{ e^{\alpha_2 t}\Pi(t)+\int_0^te^{\alpha_2 \tau}\,\Xi(\tau)\,d\tau
\leq \Pi(0)+\frac{C_{0}}{(2\alpha_2-\alpha^*)^2}+\frac{C_{0}}{(2\alpha_2-\alpha_0)^2}\hfill\cr\hfill
+C_{0}
e^{\alpha_2 t}\Bigl(\frac{\nu^{-1} e^{-\frac{\alpha_0}2t}}{2\alpha_2-\alpha_0}
+\frac{te^{-\frac{\alpha^*}2t}}{\abs{2\alpha_2-\alpha^*}} + \frac{\nu^{-2}te^{-\frac{\alpha_0}2t}}{2\alpha_2-\alpha_0}\Bigr)\cdotp}
$$
At this stage, it suffices to use the fact  that for all $k>0,$ we have
\begin{equation}\label{eq:boundk}
\sup_{t\in\R_+} te^{-kt}\leq (ek)^{-1}\andf
\int_0^\infty te^{-kt}\,dt=k^{-2},\end{equation}
and to argue as  for proving \eqref{eq:decayH1}  to  get the desired inequalities.  
\end{proof}

Granted with Proposition \ref{p:timedecay1}, one can obtain the following time decay estimates for the convective derivative: 
\begin{proposition}\label{p:timedecay2} 
Let $(\rho,u)$ be a solution to (CNS) given  by Theorem \ref{themc1}. Then, 
$$
\underset{t\in \R_+}{\sup}\int_{\T^2}te^{\alpha_3 t}\rho \abs{\t u}^2\dbar x+\int_0^\infty\!\!\int_{\T^2}te^{\alpha_3 t}\abs{\nabla \t u}^2\dbar x\,dt+\nu\int_0^\infty\!\!\int_{\T^2}te^{\alpha_3 t}\abs{\div\t u}^2\dbar x\,dt\leq C_{0}$$
with  $\alpha_3=\delta/4+\min(\alpha_2, \alpha_0/4)/8.$
\end{proposition}
\begin{proof} In what follows, we denote $A:B=\sum_{i,j}A_{ij}B_{i,j}$ if $A$ and $B$ are two $d\times d$ matrices, $(Du)_{i,j}:=\d_j u^i$ and $(\nabla u)_{ij}:=\d_i u^j$ for $1\leq i,j\leq d.$
Finally, both $\t f$ and $\frac{D}{Dt} f$  designate the convective derivative of $f.$
\medbreak
Now, applying $\frac{D}{Dt}$ to the momentum equation and testing by $te^{2\beta t}\t u,$ we obtain
\begin{equation}
 \label{eq:Ddotu}   
\int_{\T^2}\biggl(\frac{D}{Dt}(\rho\t u)-\frac{D}{Dt}\Delta u-(\lambda+1)\frac{D}{Dt}\nabla \div u+\frac{D}{Dt}\nabla P\biggr)\cdot e^{2\beta t}t \t u\dbar x=0.\end{equation}
First, taking  advantage of the mass conservation equation, we get
\begin{equation*}
\int_{\T^2} \frac{D}{Dt}(\rho \t u) e^{2\beta t}t \t u\dbar x=\frac 12 \int_{\T^2}\biggl(\frac{D}{Dt}(te^{2\beta t}\rho \abs{\t u}^2)-e^{2\beta t}(\rho \abs{\t u}^2+2\beta t\rho \abs{\t u}^2+t \div u \,\rho\abs{\t u}^2) \biggr)\dbar x.
\end{equation*}
Integrating by parts in the first term gives
\begin{multline}\label{eq:Ddotu0}
\int_{\T^2} \frac{D}{Dt}(\rho \t u) e^{2\beta t}t \t u\dbar x=\frac 12\frac{d}{dt}\int_{\T^2}te^{2\beta t}\rho \abs{\t u}^2\dbar x-\int_{\T^2}\ te^{2\beta t} \div u \, \rho\abs{\t u}^2\dbar x\\
\qquad-\frac 12\int_{\T^2} e^{2\beta t}\rho \abs{\t u}^2\dbar x-\int_{\T^2}\beta t e^{2\beta t}\rho \abs{\t u}^2\dbar x.
\end{multline}
To handle the other terms of \eqref{eq:Ddotu},  we need  the following relations:
\begin{equation*}
\begin{aligned}
&\frac{D}{Dt}\Delta u=\div(\nabla \t u-\nabla u\cdot \nabla u) -\nabla u\cdot \nabla^2 u\with  (\nabla u\cdot \nabla^2 u)^i:=\sum_{1\leq j,k\leq d} \d_k u^j \d_j\d_ku^i,\\
&\frac{D}{Dt}\nabla \div u=\nabla \div(\t u)-\nabla u\cdot \nabla \div u-\nabla(\rm{Tr}(\nabla u\cdot \nabla \q u)),\\
&\frac{D}{Dt}\nabla P=\nabla \dot P-\nabla u\cdot \nabla P \andf \t P=-(h+P)\div u.
\end{aligned}
\end{equation*}
Thanks to them and to  \eqref{eq:Ddotu0}, we get 
\begin{multline}\label{eq:esttu}
\frac 12\frac{d}{dt}\int_{\T^2}te^{2\beta t}\rho \abs{\t u}^2\dbar x+\int_{\T^2}te^{2\beta t}\abs{\nabla \t u}^2\dbar x+(\lambda+1)\int_{\T^2}te^{2\beta t}\abs{\div\t u}^2\dbar x\\
=\int_{\T^2}\ te^{2\beta t} \div u \, \rho\abs{\t u}^2\dbar x+\frac 12\int_{\T^2} e^{2\beta t}\rho \abs{\t u}^2\dbar x+\int_{\T^2}\beta t e^{2\beta t}\rho \abs{\t u}^2\dbar x\\
 + \int_{\T^2} te^{2\beta t}(\nabla u\cdot\nabla u)\cdot \nabla \t u\dbar x+ \int_{\T^2} te^{2\beta t}(\nabla u\cdot\nabla^2 u)\cdot \t u\dbar x\\
 +(\lambda+1) \int_{\T^2} te^{2\beta t}{\rm Tr}(\nabla u\cdot \nabla \q u) \, \div \t u\dbar x+(\lambda+1) \int_{\T^2} te^{2\beta t}(\nabla u\cdot\nabla\div u)\, \t u\dbar x\\
 -\int_{\T^2} te^{2\beta t}(h+P)\div u\cdot \div \t u\dbar x+\int_{\T^2} te^{2\beta t}(\nabla u\cdot \nabla \wt P) \cdot \t u\dbar x=:\sum^{9}_{k=1}I_k.
\end{multline}
Before estimating the terms $I_j$, we present some results that play an important role in the proof. At first, we 
claim that for, say, $\eta=\delta/2+\min(\alpha_2/4,\alpha_0/16),$
we have 
\begin{equation}\label{eq:tul4l4}
    \|t^{1/2}e^{\eta t}\nabla \p u\|_{L_4(\R_+\times\T^2)}\leq  C_{0}\andf \|t^{1/4}e^{\eta t}\div u\|_{L_4(\R_+\times\T^2)}\leq C_{0}\nu^{-\frac{1}{2}}.
\end{equation}
The first inequality stems from  \eqref{eq:tweug1} and \eqref{eq:twel2l2}
and the fact that \begin{equation*}
\|t^{1/2}e^{\eta t}\nabla \p u\|^4_{L_4(\R_+\times\T^2)}\lesssim \int_0^{\infty} e^{4\eta t} \|t^{1/2}\nabla \p u\|^2_{L_2}\|t^{1/2}\nabla^2 \p u\|^2_{L_2}\,dt.
\end{equation*}
Similarly, we can get (using \eqref{eq:edl2h1} and \eqref{eq:decayH1}to obtain the second inequality)
\begin{equation}\label{eq:tgl4l4}
    \|t^{1/2}e^{\eta t}\rG\|_{L_4(\R_+\times\T^2)}\leq C_{0} \nu^{{1}/{4}}\andf \|t^{1/4}e^{\eta t}\rG\|_{L_4(\R_+\times\T^2)}\leq C_{0} \nu^{{1}/{4}}
\end{equation}
and using \eqref{eq:decayP}, \eqref{eq:boundk}  gives
\begin{equation}\label{eq:tpl4l4}
\|t^{1/4}e^{\eta t}\rP\|^4_{L_4(\R_+\times\T^2)}\lesssim \int_0^{\infty} te^{4\eta t} \|\rP\|^2_{L_2}\|\rP\|^2_{L_\infty}\,dt\leq C_{0} \nu^2.
\end{equation}
Putting \eqref{eq:tgl4l4} and \eqref{eq:tpl4l4} together, and remembering \eqref{eq:divurprg} gives the second part
of \eqref{eq:tul4l4}. 
Note that arguing similarly but using only  \eqref{eq:decayH1} and \eqref{eq:divurprg} , 
one  can easily get  
\begin{equation}\label{eq:ul4l4}
    \|e^{\eta t}\nabla u\|_{L_4(\R_+\times\T^2)}\leq C_0\andf \|e^{\eta t}\div u\|_{L_4(\R_+\times\T^2)}\leq C_0\nu^{-\frac{3}{4}}.
\end{equation}
Finally, we shall also use the fact that, 
owing to Relation \eqref{eq:ud2}, we have for all $p\in(1,\infty),$
\begin{equation}\label{eq:Lp}\|\nabla G\|_{L_p}+\|\nabla^2 \p u\|_{L_p}\leq C\|\rho \dot u\|_{L_p}.\end{equation}
We are ready to bound the terms $I_j$. 
Regarding  $I_1,$ we observe  that since $\rho\t u$ has mean value zero, 
Inequality \eqref{eq:sobolev} and Theorem \ref{themc1}   allow to write that
\begin{align}
I_1= \int_{\T^2}\ te^{2\beta t} \rho\,\div u  \,\abs{\t u}^2\dbar x
&\leq te^{2\beta t}\rho^*\|\div u\|_{L_2}\|\t u\|^2_{L_4}\nonumber\\
&\leq C_{0}\nu^{-\frac{1}{2}}
\|\sqrt{t}e^{\beta t}\nabla \t u\|^2_{L_2}.
\end{align}
So, if we assume that $\nu\geq (\mu^{-1} 4 C_{0})^2,$ then one gets
\begin{equation}\label{eq:I1}
I_1\leq \frac{\mu}{16}\int_{\T^2}\|\sqrt{t}e^{\beta t}\nabla \t u\|^2_{L_2}\,\dbar x.\end{equation}
Next, combining \eqref{eq:ul4l4} with \eqref{eq:sobolev} and  \eqref{eq:edl2h1}  yields for all $T>0$ (if $\beta\leq2\eta$):
\begin{align}\label{eq:I2}
    \int_0^T I_2\,dt&\leq \int_0^T
    e^{2(\beta-\alpha_2)t}\|e^{\alpha_2 t}\sqrt{\rho}u_t\|^2_{L_2}\,dt+\int_0^T \|e^{\beta t}\sqrt{\rho}u \cdot \nabla u\|^2_{L_2}\,dt\nonumber\\
    &\leq C_{0}+\int_0^T e^{2\beta t}\|\sqrt{\rho}u \|^2_{L_4}\|\nabla u \|^2_{L_4}\,dt\nonumber\\
    &\leq C_{0}+\int_0^T e^{(2\beta-4\eta)t}
    \|e^{\eta t}\nabla u \|^4_{L_4}\,dt\nonumber\\
    &\leq C_{0}.
\end{align}
Next, in light of \eqref{eq:twel2l2}, we have 
\begin{equation}\label{eq:I3}
\int_0^T I_3\,dt\leq C_{0}\quad\hbox{for all }\ T>0.\end{equation}
For $I_4,$ using H\"older inequality, \eqref{eq:interl4d}, \eqref{eq:eqg2u} and 
\eqref{eq:denu} we have
\begin{equation*}
\begin{aligned}
    I_4&\lesssim \|\sqrt{t}e^{\beta t}\nabla \t u\|_{L_2}\sqrt{t}e^{\beta t}
    \Bigl(\|\nabla \p u\|_{L_{4}}+\frac{1}{\nu}\|\rP\|_{L_{4}}+\frac{1}{\nu}\|\rG\|_{L_{4}}\Bigr)^2\\
    &\lesssim \|\sqrt{t}e^{\beta t}\nabla \t u\|_{L_2}\sqrt{t}e^{\beta t}\Bigl(\|\nabla \p u\|_{L_2}\|\nabla^2 \p u\|_{L_2}+\frac{1}{\nu^2}\|\rP\|_{L_2}\|\rP\|_{L_\infty}+\frac{1}{\nu^2}\|\rG\|_{L_2}\|\nabla G\|_{L_2}\Bigr)\\
    &\lesssim  \|\sqrt{t}e^{\beta t}\nabla \t u\|_{L_2}\biggl(\Bigl(\|\nabla \p u\|_{L_2}+\frac{1}{\nu^2}\|\rG\|_{L_2}\Bigr)\|\sqrt{t}e^{\beta t}\rho \t u\|_{L_2}+\frac{1}{\nu^2}\sqrt{t}e^{\beta t}\|\rP\|_{L_2}\|\rP\|_{L_\infty}\biggr)\\
    &\leq  \frac{1}{16} \|\sqrt{t}e^{\beta t}\nabla \t u\|^2_{L_2}+C 
    \biggl(\Bigl(\|\nabla \p u\|_{L_2}\!+\!\frac{1}{\nu^2}\|\rG\|_{L_2}\Bigr)\|\sqrt{t}e^{\beta t}\rho \t u\|_{L_2}\!+\!\frac{\sqrt{t}e^{\beta t}}{\nu^2}\|\rP\|_{L_2}\|\rP\|_{L_\infty}\biggr)^2
    \cdotp
\end{aligned}
\end{equation*}
Hence, by using \eqref{eq:decayH1} and \eqref{eq:boundk}, and  choosing $\beta <\alpha_0/5,$ we obtain for all $T>0$
\begin{align}\label{eq:I4}
    \int_0^T I_4\,dt
    &\leq  \frac{1}{16} \int_0^T\|\sqrt{t}e^{\beta t}\nabla \t u\|^2_{L_2}\,dt\nonumber\\
    +C \int_0^T &\biggl((\|\nabla \p u\|_{L_2}+\frac{1}{\nu^2}\|\rG\|_{L_2})^2\|\sqrt{t}e^{\beta t}\rho \t u\|^2_{L_2}+\frac{C_{0}}{\nu^3}te^{2\beta t}\Bigl(e^{-\frac{\alpha^*}2 t}+\nu^{-1}e^{-\frac{\alpha_0}2t}\Bigr)\biggr)dt\nonumber\\
     &\leq  \frac{1}{16} \int_0^T\!\!\|\sqrt{t}e^{\beta t}\nabla \t u\|^2_{L_2}\,dt
    +C \int_0^T\!\! \biggl((\|\nabla \p u\|_{L_2}\!+\!\frac{1}{\nu^2}\|\rG\|_{L_2})^2\|\sqrt{t}e^{\beta t}\rho \t u\|^2_{L_2}\biggr)dt\!+\! \frac{C_{0}}{\nu^2}\nonumber\\&\leq  \frac{1}{16} \int_0^T\!\!\|\sqrt{t}e^{\beta t}\nabla \t u\|^2_{L_2}\,dt
    +C_{0}\int_0^T
\bigl(e^{-\frac{\alpha^*}2t}+\nu^{-1}e^{-\frac{\alpha_0}2t}\bigr)\|\sqrt{t}e^{\beta t}\rho \t u\|^2_{L_2}\,dt +\nu^{-2}C_{0}.
\end{align}
For $I_6,$ since $\nu \q u=-\nabla (-\Delta)^{-1}(\rG+\rP),$ and we have \eqref{eq:interl4d},  \eqref{eq:eqg2u}, we find that
$$\begin{aligned}
I_6
&\lesssim \frac{(\lambda\!+\!1)}{\nu}\|\sqrt{t}e^{\beta t}\div \t u\|_{L_2}\sqrt{t}e^{\beta t}\Bigl(\|\nabla \p u\|_{L_{4}}\|\rG+\rP\|_{L_4}+\frac{1}{\nu}(\|\wt G\|^2_{L_4}+\|\wt P\|^2_{L_4})\Bigr)\\
&\lesssim\frac{(\lambda\!+\!1)}{\nu}\|\sqrt{t}e^{\beta t}\div \t u\|_{L_2}\Bigl( \|\nabla \p u\|^{\frac 12}_{L_2}\|\rG\|^{\frac 12}_{L_2}\|\sqrt{t}e^{\beta t}\rho \t u\|_{L_2}\!+\!\sqrt{t}e^{\beta t}\frac{1}{\nu}\|(\rG,\rP)\|^2_{L_4}\Bigr)\\
&\hspace{2cm}+\frac{(\lambda\!+\!1)}{\nu}\|\sqrt{t}e^{\beta t}\div \t u\|_{L_2}\sqrt{t}e^{\beta t}\|\nabla \p u\|^{\frac 12}_{L_2}\|\rho \t u\|^{\frac 12}_{L_2}\|\rP\|^{\frac 12}_{L_2}\|\rP\|^{\frac 12}_{L_\infty}\\
&\leq  C\frac{(\lambda\!+\!1)}{\nu^2}
\Bigl( \|\nabla \p u\|_{L_2}\|\rG\|_{L_2}\|\sqrt{t}e^{\beta t}\rho \t u\|^2_{L_2} +\frac{1}{\nu^2}\|t^{1/4}e^{\frac{\beta}2t}(\rG,\rP)\|^4_{L_4}\\
&\hspace{2cm}+te^{2\beta t}\|\nabla \p u\|_{L_2}\|\rho \t u\|_{L_2}\|\rP\|_{L_2}\|\rP\|_{L_\infty}\Bigr)
+\frac{(\lambda\!+\!1)}{4}\|\sqrt{t}e^{\beta t}\div \t u\|^2_{L_2}.
\end{aligned}$$
Hence, according to  \eqref{eq:decayH1}, \eqref{eq:tgl4l4},  \eqref{eq:tpl4l4},
\eqref{eq:decayP} and \eqref{eq:tweug1}, choosing $\beta <2\eta$ we have 
\begin{align}\label{eq:I6}
     \int_0^T \!I_6\,dt&\leq   C\nu^{-1}\!\int_0^T\!
\Bigl(( \|\nabla \p u\|_{L_2}\|\rG\|_{L_2}\!+\!\|\rP\|_{L_2}^2)\|\sqrt{t}e^{\beta t}\rho \t u\|^2_{L_2}\!+\!\|\sqrt te^{\beta t}\nabla \p u\|^2_{L_2}\|\rP\|_{L_\infty}^2\Bigr)dt\nonumber\\\nonumber
&\hspace{2.3cm}+C\nu^{-3}\int_0^T \|t^{1/4}e^{\frac{\beta}2t}(\rG,\rP)\|^4_{L_4}\,dt
+\frac{(\lambda\!+\!1)}{4}\int_0^T\|\sqrt{t}e^{\beta t}\div \t u\|^2_{L_2}\,dt\\
&\leq  C_{0}\nu^{-1}\int_0^T \Bigl(\sqrt\nu e^{-\frac{\alpha^*}2t}+
e^{-\frac{\alpha_0}2t}\Bigr)\|\sqrt{t}e^{\beta t}\rho \t u\|^2_{L_2}\,dt\nonumber
\\&\hspace{2.3cm}+\nu^{-1}C_{0}
+\frac{(\lambda\!+\!1)}{4}\int_0^\infty\|\sqrt{t}e^{\beta t}\div \t u\|^2_{L_2}\,dt.
\end{align}
Thanks to \eqref{eq:decayH1}, we obtain
$$\begin{aligned}
I_8&\leq \|\sqrt{t}e^{\beta t}\div \t u\|_{L_2}\sqrt{t}e^{\beta t}\|\div u\|_{L_2}\|h+P\|_{L_\infty}\\
&\leq \frac{(\lambda+1)}{16}\|\sqrt{t}e^{\beta t}\div \t u\|^2_{L_2}+\frac{C}{(\lambda+1)}te^{2\beta t}\|\div u\|^2_{L_2}\|h+P\|^2_{L_\infty}\\
&\leq \frac{(\lambda+1)}{16}\|\sqrt{t}e^{\beta t}\div \t u\|^2_{L_2}+\frac{C_{0}}{\nu^2}te^{2\beta t}\Bigl(e^{-\frac{\alpha^*}2 t}+\nu^{-1}e^{-\frac{\alpha_0}2t}\Bigr)\cdotp
\end{aligned}$$
Hence for all $T>0$,
\begin{equation}\label{eq:I8}
\int_0^T I_8\,dt \leq  \frac{(\lambda+1)}{16}\int_0^T \|\sqrt{t}e^{\beta t}\div \t u\|^2_{L_2}\,dt+\frac{C_{0}}{\nu}\cdotp\end{equation}
Owing to \eqref{eq:divurprg} and to the decomposition \eqref{eq:denu}, we find that  
    \begin{multline}\label{eq:I70}
        I_7=\frac{(\lambda\!+\!1) }{\nu} I_9+I_{71}\\
                \with I_{71}:=\frac{(\lambda\!+\!1) }{\nu}
        \int_{\T^2} te^{2\beta t}\Bigl(\Bigl(\nabla \p u-\frac{1}{\nu}\nabla^2(-\Delta)^{-1}(\rG+\rP)\Bigr)\cdot\nabla G\Bigr)\cdot \t u\dbar x\biggr)\cdotp
    \end{multline}
Remembering that $\lambda+1\leq\nu$ and using \eqref{eq:Lp}, \eqref{eq:sobolev}, \eqref{eq:pti} and \eqref{eq:decayH1}, we get 
$$\begin{aligned}
  I_{71}&\leq t e^{2\beta t}\|\nabla \p u\|_{L_2}\|\nabla G\|_{L_4}\|\t u\|_{L_4}+\frac{1}{\nu} t e^{2\beta t}(\|\rG\|_{L_2}+\|\rP\|_{L_2})\|\nabla G\|_{L_4}\|\t u\|_{L_4}\\
   &\leq  t e^{2\beta t}\|\nabla \p u\|_{L_2}\|\rho\t u\|^{\frac 14}_{L_2}\|\rho\t u\|^{\frac 34}_{L_6}\|\t u\|_{L_4}+\frac{C_{0}}{\nu^{1/2}} t e^{2\beta t}\|\nabla \t u\|^2_{L_2}\\
   &\lesssim C_{0}\|\nabla \p u\|_{L_2}\|\sqrt{t}e^{\beta t}\rho \t u\|^{\frac 14}_{L_2}\|\sqrt{t}e^{\beta t}\nabla \t u\|^{\frac 74}_{L_2}+\frac{C_{0}}{\nu^{1/2}} t e^{2\beta t}\|\nabla \t u\|^2_{L_2}\\
    &\leq\frac{1}{32}\|\sqrt{t}e^{\beta t}\nabla \t u\|^2_{L_2}+
    C_{0}\|\nabla \p u\|^8_{L_2}\|\sqrt{t}e^{\beta t}\rho \t u\|^2_{L_2}+
    \frac{C_{0}}{\nu^{1/2}} t e^{2\beta t}\|\nabla \t u\|^2_{L_2}.
\end{aligned}$$
Hence, assuming that  $\nu^{1/2}>32 C_{0},$ one gets 
\begin{equation}\label{eq:I71}
\int_0^T I_{71} \,dt\leq  \frac{1}{16}\int_0^T \|\sqrt{t}e^{\beta t}\nabla \t u\|^2_{L_2}\,dt+C_{0}\int_0^T \|\nabla \p u\|^8_{L_2}\|\sqrt{t}e^{\beta t}\rho \t u\|^2_{L_2}\,dt.\end{equation}
To handle $I_9,$ we use again  \eqref{eq:divurprg}, integrate by parts 
where needed and use \eqref{eq:eqg2u} and \eqref{eq:decayH1} to get: 
$$\begin{aligned}
    I_9&=-\int_{\T^2} te^{2\beta t}\tilde P \dot u\cdot\nabla\div  u\dbar x
    -\int_{\T^2} te^{2\beta t}\tilde P \nabla u:D\t u \dbar x   \\
     &=\frac{1}{2\nu} \int_{\T^2} te^{2\beta t} \rP^2 \div \t u \dbar x-\frac{1}{\nu} \int_{\T^2} te^{2\beta t}\rP \nabla G\cdot \t u \dbar x-\int_{\T^2} te^{2\beta t}\rP \nabla u:D\t u \dbar x\\
    & =:I_{91}+I_{92}+I_{93}.
     \end{aligned}$$
     From \eqref{eq:pti}, H\"older and  Young  inequality, we infer that
     $$\begin{aligned}
    I_{91}+I_{92} 
     \leq& \frac{\sqrt{t}}{2\nu}e^{\beta t}\|\rP\|_{L_2}\|\rP\|_{L_\infty}\|\sqrt{t}e^{\beta t}\div \t u\|_{L_2}\\
     &\hspace{2cm}+\nu^{-1}\|\sqrt{t}e^{\beta t}\nabla \t u\|_{L_2}\sqrt{t}e^{\beta t}\|\rP\|_{L_\infty}\|\nabla G\|_{L_2}\\
    \leq  & \frac{(\lambda+1)}{16}\|\sqrt{t}e^{\beta t}\div \t u\|^2_{L_2}+\frac{1}{32} \|\sqrt{t}e^{\beta t}\nabla \t u\|^2_{L_2}\\
      &\hspace{2cm}+C\nu^{-2}te^{2\beta t}\Bigl(\nu^{-1}\|\rP\|^2_{L_2}\|\rP\|^2_{L_\infty}+\|\rP\|^2_{L_\infty}\|\nabla G\|^2_{L_2} \Bigr)\cdotp
\end{aligned}$$
     Hence, taking advantage of \eqref{eq:edl2h1},  \eqref{eq:boundk} and  \eqref{eq:decayP}, we get if  $\beta<\alpha_2,$
     $$ \int_0^T  (I_{91}+I_{92}) \,dt\leq\frac{(\lambda+1)}{16}\int_0^T \|\sqrt{t}e^{\beta t}\div \t u\|^2_{L_2}\,dt+\frac{1}{32}\int_0^T  \|\sqrt{t}e^{\beta t}\nabla \t u\|^2_{L_2}\,dt
      +C_{0}\nu^{-1}.$$
       Since $\alpha_2\simeq\nu^{-1},$ 
using \eqref{eq:denu}, \eqref{eq:edl2h1} and \eqref{eq:boundk} gives
$$\begin{aligned}
\int_0^T I_{93}\,dt&\leq  \int_0^T \|\sqrt{t}e^{\beta t}\nabla \t u\|_{L_2}\sqrt{t}e^{\beta t} \|\rP \cdot \nabla u\|_{L_2} \,dt\\
&\leq \frac{1}{32} \int_0^T \|\sqrt{t}e^{\beta t}\nabla \t u\|^2_{L_2} \,dt+C\int_0^T \|\rP\|_{L_\infty}^2 \|\sqrt te^{\beta t}\nabla \p u\|^2_{L_2} \,dt\\
&\hspace{2cm}+\frac{C}{\nu^2}\int_0^T te^{2\beta t}\|\rP\|^4_{L_4} \,dt+ \frac{C}{\nu^2}\int_0^T te^{2\beta t}\|\rP\cdot \rG\|^2_{L_2} \,dt.\end{aligned}$$
      The last two terms may be bounded independently of $\nu$ 
(provided $\beta\leq2\eta$) by means of \eqref{eq:tul4l4} and \eqref{eq:tgl4l4}. Next,
thanks to \eqref{eq:twel2l2} and to \eqref{eq:poinca}, we have
$$\int_0^T t e^{2\beta t} \|\rP\|^2_{L_\infty}\|\nabla \p u\|^2_{L_2} \,dt
    \leq  C_{\rho^*}\int_0^T t e^{2\beta t}\|\nabla^2 \p u\|^2_{L_2} \,dt \leq  C_{0}.$$
Hence, remembering \eqref{eq:I70} and \eqref{eq:I71},
\begin{equation}\label{eq:I9}
\int_0^T (I_7+I_{9})\,dt \leq \frac{(\lambda+1)}{8}\int_0^T \|\sqrt{t}e^{\beta t}\div \t u\|^2_{L_2}\,dt+\frac{3}{16}\int_0^T \|\sqrt{t}e^{\beta t}\nabla \t u\|^2_{L_2}\,dt+C_{0}.\end{equation}
      Finally, using again \eqref{eq:denu}, we obtain
$$
\begin{aligned}
I_5&= \int_{\T^2} te^{2\beta t}\biggl(\nabla u\cdot\nabla^2 \Bigl(\p u-\frac{1}{\nu}\nabla(-\Delta)^{-1}\rG-\frac{1}{\nu}\nabla(-\Delta)^{-1}\rP\Bigr)\biggr)\cdot \t u\dbar x\\
&=:I_{51}+I_{52}+I_{53}.
\end{aligned}$$ 
By  continuity of Riesz operators on $L_4$ and H\"older inequality, one gets
$$I_{51}+I_{52}\leq C te^{2\beta t}\|\nabla u\|_{L_2}(\|\nabla^2 \p u\|_{L_4}+\|\nabla G\|_{L_4})\|\t u\|_{L_4}.$$
Hence, remembering \eqref{eq:sobolev} and arguing as for proving \eqref{eq:esucl4b}, we easily get
$$\begin{aligned}
I_{51}+I_{52}&\leq  C_{\rho^*} 
\sqrt{t}e^{\beta t}\|\nabla  u\|_{L_2}\|\rho^{1/4}\t u\|_{L_4}\|\sqrt{t}e^{\beta t}\nabla \t u\|_{L_2},\\
&\leq  C_{\rho^*} \|\sqrt{t}e^{\beta t}\nabla \dot u\|_{L_2}^{7/4}
\|\sqrt{\rho t} e^{\beta t}\t u\|_{L_2}^{1/4}\|\nabla u\|_{L_2}\\
&\leq \frac{1}{16}\|\sqrt{t}e^{\beta t}\nabla \t u\|^{2}_{L_2} +  C_{\rho^*} \|\nabla u\|_{L_2}^8\|\sqrt{\rho t} e^{\beta t}\t u\|_{L_2}^{2}\\
&\leq \frac{1}{16}\|\sqrt{t}e^{\beta t}\nabla \t u\|^{2}_{L_2} +  C_{0} \Bigl(e^{-\frac{\alpha^*}2 t}+\nu^{-1}e^{-\frac{\alpha_0}2t}\Bigr)^4\|\sqrt{\rho t} e^{\beta t}\t u\|_{L_2}^{2}
\end{aligned}$$
      which entails
\begin{equation}\label{eq:I51}\int_0^T (I_{51}+I_{52})\,dt \leq \frac{1}{16}\|\sqrt{t}e^{\beta t}\nabla \t u\|^{2}_{L_2} + C_{0}\!\int_0^T\! \Bigl(e^{-2\alpha^* t}+\frac{e^{-2\alpha_0t}}{\nu^4}\Bigr)\|\sqrt{\rho t} e^{\beta t}\t u\|_{L_2}^{2}\,dt.\end{equation}
To bound $I_{53}$, we set $\phi:=(-\Delta)^{-1}\rP$ and use the following identity (see \cite[(B.6)]{DM-ripped}):
$$\begin{aligned}
    I_{53}
&=\frac{1}{2\nu^2}\int_{\T^2}t e^{2\beta t}\rP^2\cdot \div \t u\dbar x-\frac{1}{\nu^2}\int_{\T^2}t e^{2\beta t}\rP \cdot \nabla^2\phi\cdot \nabla \t u\dbar x\\
&\hspace{2cm}+\frac{1}{\nu^2}\int_{\T^2}t e^{2\beta t} \nabla G\cdot \nabla^2\phi\cdot \t u\dbar x+\frac{1}{\nu}\int_{\T^2}t e^{2\beta t}\nabla u\cdot \nabla^2\phi\cdot \nabla \t u\dbar x.
\end{aligned}$$
Therefore, thanks to H\"older inequality, \eqref{eq:pti} and to the continuity of $\nabla^2(-\Delta)^{-1}$ on $L_4(\T^2),$ we get 
$$\begin{aligned}
I_{53} 
&\leq \frac{1}{\nu^2}\sqrt{t}e^{\beta t}\|\rP\|^2_{L_4}(\frac 12\|\sqrt{t}e^{\beta t}\div \t u\|_{L_2}+\|\sqrt{t}e^{\beta t}\nabla \t u\|_{L_2})\\
&+\frac{1}{\nu^2}\sqrt{t}e^{\beta t}(\|\nabla G\|_{L_4}\|\rP\|_{L_4}+\nu\|\nabla u\|_{L_4}\|\rP\|_{L_4})\|\sqrt{t}e^{\beta t}\nabla \t u\|_{L_2}\\
\leq &\frac{(\lambda+1)}{16}\|\sqrt{t}e^{\beta t}\div \t u\|^2_{L_2}+\frac{1}{16} \|\sqrt{t}e^{\beta t}\nabla \t u\|^2_{L_2}\\
&+\frac{C}{\nu^4}\biggl(te^{2\beta t}\|\rP\|^4_{L_4}+te^{2\beta t}(\|\nabla G\|^2_{L_4}\|\rP\|_{L_2}
\|\rP\|_{L_\infty}+\nu^2\|\nabla u\|^2_{L_4}\|\rP\|^2_{L_4})\biggr)\cdotp
\end{aligned}$$
Hence, using  \eqref{eq:edl2h1}, \eqref{eq:tul4l4}, \eqref{eq:tpl4l4},
\eqref{eq:boundk},  and choosing $\beta<\eta,$ we conclude that
\begin{align}\label{eq:I53}
    \int_0^T I_{53}\,dt&\leq \frac{(\lambda+1)}{16}\int_0^T\|\sqrt{t}e^{\beta t}\div \t u\|^2_{L_2}\,dt+\frac{1}{16} \int_0^T\|\sqrt{t}e^{\beta t}\nabla \t u\|^2_{L_2}\,dt\nonumber\\
&\qquad\qquad+\frac{C}{\nu^4}\int_0^T te^{2(\beta-\eta) t}\|e^{\eta t}\rP\|^4_{L_4}\,dt\nonumber\\
&\qquad\qquad+\frac{C}{\nu^4}\int_0^T te^{2(\beta-\alpha_2) t}\|e^{\alpha_2t}\nabla G\|^2_{L_2}\|\rP\|_{L_2}\|\rP\|_{L_\infty}\,dt\nonumber\\
&\qquad\qquad+ \frac{C}{\nu^2}\int_0^T e^{2(\beta-2\eta) t}\|t^{1/4}e^{\eta t}\nabla u\|^2_{L_4}\|t^{1/4}e^{\eta t}\rP\|^2_{L_4}\,dt\nonumber\\
&\leq \frac{(\lambda+1)}{16}\int_0^T\|\sqrt{t}e^{\beta t}\div \t u\|^2_{L_2}\,dt+\frac{1}{16} \int_0^T\|\sqrt{t}e^{\beta t}\nabla \t u\|^2_{L_2}\,dt
+\frac{C_{0}}\nu\cdotp
\end{align}
Finally, plugging Inequalities \eqref{eq:I1}, \eqref{eq:I2}, 
\eqref{eq:I3}, \eqref{eq:I4}, \eqref{eq:I6}, \eqref{eq:I8}, \eqref{eq:I9}, 
 \eqref{eq:I51} and \eqref{eq:I53} in \eqref{eq:esttu}
(after integrating on $[0,t]$), 
we end up with 
\begin{multline*}
 \|e^{\beta t}\sqrt{\rho t}\,\t u(t)\|^2_{L_2}
 +\frac{1}{2}\int_0^t\|\sqrt \tau\,e^{\beta \tau}\nabla \t u\|_{L_2}^2\,d\tau+\frac{\nu}{2}\int_0^t \|\sqrt\tau\,e^{\beta \tau}\divv \t u\|_{L_2}^2\,d\tau\\
    \leq C_{0}\biggl(1 +\int_0^t\bigl(\nu^{-1/2}e^{-\frac{\alpha^*}2 \tau}+\nu^{-1}e^{-\frac{\alpha_0}2\tau}\bigr)
    \|e^{\beta \tau}\sqrt{\rho \tau}\,\t u\|^2_{L_2}\,d\tau\biggr)\cdotp
\end{multline*}
Then, using Gronwall and taking advantage of \eqref{eq:boundk}
completes the proof. 
\end{proof}

\begin{corollary}\label{c:decayG}
Let $(\rho,u)$ be a solution to (CNS) given by Theorem \ref{themc1}.  
There exists two positive constants $C_{0}$ and $\nu_0$ depending only on the initial data, $\T^2,$ $\mu$, $\kappa$ and $\gamma,$ positive constant $c,$ such that if $\nu\geq\nu_0$  then  
\begin{equation*}
\int_0^\infty e^{\frac{\alpha_3}4t}\|\rG\|_{L_\infty}\,dt\leq C_{0}\,\nu^{c\theta}\andf\int_0^\infty e^{\frac{\alpha_3}4t}\|\nabla \p u\|_{L_\infty}\,dt\leq C_{0}.
\end{equation*}
\end{corollary}
\begin{proof}
We start with the Gagliardo-Nirenberg inequality 
$$\|\wt G\|_{L_\infty} \lesssim \|\wt G\|_{L_2}^\theta\|\nabla G\|^{1-\theta}_{L_p}\with \theta:=\frac{p-2}{2p-2}
$$
which, combined with Inequalities \eqref{eq:Lp} and \eqref{eq:sobolev} implies that 
$$ \|\wt G\|_{L_\infty} \leq C_0 \|\wt G\|_{L_2}^\theta\|\nabla\dot u\|^{1-\theta}_{L_2}.$$
Hence we have, using H\"older inequality to get the second line,
$$\begin{aligned}\int_0^\infty\!\! e^{(1-\theta)\frac{\alpha_3}2t} \|\wt G\|_{L_\infty}\,dt&\leq C_0\nu^{\frac\theta2}
\int_0^\infty  \biggl(\frac{\, \|\wt G\|_{L_2}}{\nu^{1/2}}\biggr)^\theta\,\bigl\|e^{\frac{\alpha_3}2t}\sqrt t\nabla\dot u\bigr\|^{1-\theta}_{L_2}\,\frac{dt}{t^{\frac{1-\theta}2}}\\
&\leq C_0\nu^{\frac\theta2}
\biggl(\int_0^\infty \! \biggl(\frac{ \|\wt G\|_{L_2}^2}{\nu}\biggr)^{\frac{\theta}{1+\theta}}\frac{dt}{t^{\frac{1-\theta}{1+\theta}}}\biggr)\biggr)^{\frac{1+\theta}2}
\biggl(\int_0^\infty \!t\|e^{\alpha_3 t} \nabla\dot u\|^{2}_{L_2}\,dt\biggr)^{\frac{1-\theta}2}\cdotp\end{aligned}$$

The term with $\wt G$ (resp. $\nabla \dot u$) can be bounded according to \eqref{eq:decayH1} 
(resp. Proposition \ref{p:timedecay2}), which, since $\theta\in(1/2,1)$ gives the first inequality of the corollary. 
\smallbreak
For proving the second inequality, we proceed the same and now get 
$$
\int_0^\infty e^{\frac{\alpha_3}4t} \|\nabla\p u\|_{L_\infty}\,dt \lesssim 
\int_0^\infty  e^{\frac{\alpha_3}4t} \|\nabla \p u\|_{L_2}^\theta\|\nabla\dot u\|^{1-\theta}_{L_2}\,dt.$$
The only difference is that we use \eqref{eq:tweug1} to bound $\|\sqrt t\nabla \p u\|_{L_2}.$ 
This allows to complete the proof of the corollary. 
\end{proof}



\subsection* {Acknowledgments:}
 The second author has been partly funded by the B\'ezout Labex, funded by ANR,  reference ANR-10-LABX-58.


\appendix
\section{}
Let us first prove Poincar\'e and Sobolev inequalities in the torus
for functions that need not have zero mean value. 
\begin{proposition}\label{prop1}
Let $a$ be a measurable function on $\T^d$ with  mean value $1,$
and let $z$ be in $H^1(\T^d).$ Denote by $M_a(z)$ the mean value of $az.$
Then, we have for all $d\geq1,$
\begin{multline}\label{eq:poincare}
\|z\|_{L_2(\T^d)}^2\leq  |M_a(z)|\bigl(|M_a(z)|+2c_{\T^d}\|a-1\|_{L_2(\T^d)}\|\nabla z\|_{L_2(\T^d)}\bigr)\\
+c_{\T^d}^2 \|a\|_{L_2(\T^d)}^2\|\nabla z\|_{L_2(\T^d)}^2.
  \end{multline}
  Furthermore, in the case  $M_a(z)=0$ and $d=2,$  then we have
  for all $p\in[2,\infty)$
   \begin{align}\label{eq:GN}
  \|z\|_{L_p(\T^2)}&\leq C_{p,\T^2}
   \|a\|_{L_2}\|z\|_{L_2(\T^2)}^{2/p}\|\nabla z\|_{L_2(\T^2)}^{1-2/p},\\\label{eq:sobolev}
  \|z\|_{L_p(\T^2)}&\leq C_p c_{\T^2}^{2/p}\|a\|_{L_2(\T^2)}\|\nabla z\|_{L_2(\T^2)}.\end{align}  
\end{proposition}
\begin{proof}
Let $\tilde{z}:=z-\bar{z}$ where $\bar{z}$ stands for the mean value of $z$. 
Then, we have
\begin{equation}\label{eq:poinca0}\|z\|_{L_2(\T^d)}^2=
\abs{\bar{z}}^2+\|\wt z\|_{L_2(\T^d)}^2.\end{equation}
As the mean value of $a$ is $1,$ we have
$$\int_{\T^d} a z\dbar x=\int_{\T^d} a \bar{z}\dbar x+\int_{\T^d} a \tilde{z}\dbar x= \bar{z} +\int_{\T^d} (a-1) \tilde{z}\dbar x.$$
By Cauchy-Schwarz inequality, this implies that
\begin{equation}\label{eq:meanvalue}
 \abs{\bar{z}}\leq  \abs{\int_{\T^d}az\dbar x}+\|a-1\|_{L_2}\|\wt z\|_{L_2}.
    \end{equation}
The right-hand side may be bounded by means of \eqref{eq:poinca} and, since  the mean value of $a$
is $1,$ we have
$$\|a-1\|_{L_2}^2+1=\|a\|_{L_2}^2.$$
Hence, reverting to \eqref{eq:poinca0}  yields \eqref{eq:poincare}.
\medbreak
In order to prove \eqref{eq:GN}, we take advantage of 
the Gagliardo-Nirenberg inequality for functions with zero mean value
and argue as follows:
$$\begin{aligned} \|z\|_{L_p}&\leq |\bar z|+\|\wt z\|_{L_p}\\
&\leq \bigl(1+\|a-1\|_{L_{p'}})\|\wt z\|_{L_p}\\
&\leq C_p\bigl(1+|\T^2|^{\frac12-\frac1p}\|a-1\|_{L_2}\bigr)\|\wt z\|_{L_2}^{2/p}\|\nabla \wt z\|_{L_2}^{1-2/p}.
\end{aligned}$$
As $\|\wt z\|_{L_2}\leq \|z\|_{L_2}$ and $1+\|a-1\|_{L_2}\leq\sqrt 2\|a\|_{L_2},$
we get \eqref{eq:GN}. 
Proving \eqref{eq:sobolev} is similar except that we use \eqref{eq:meanvalue} to bound $\bar z,$
then  \eqref{eq:poinca}. 
\end{proof}
\begin{remark}\label{rem1}
For smooth enough solutions  of Systems $(INS)$ or $(CNS),$ we have
\begin{equation}
    \int_{\T^d} \rho u_t\dbar x=-\int_{\T^d} \rho u \cdot \nabla u\dbar x.
\end{equation}
Hence, if the mean value of $\rho$ is $1$ then
\begin{equation}\label{eq:utl2}
    \|u_t\|_{L_2}\leq \abs{\int_{\T^d} \rho u \cdot \nabla u\dbar x}+
c_{\T^d}\sqrt2\|\rho\|_{L_2} \|\nabla u_t\|_{L_2}.
\end{equation}
Similarly, as  the  mean value of $\rho u$ and of $\rho \t u$ 
(with $\t u:=u_t+u\cdot\nabla u$)
is $0,$   we have
\begin{equation}\label{eq:pti}
   \|u\|_{L_2(\T^d)}\leq c_{\T^d}\|\rho\|_{L_2}\|\nabla u\|_{L_2(\T^d)}\andf
     \|\t u\|_{L_2(\T^d)}\leq c_{\T^d}\|\rho\|_{L_2}\|\nabla \t u\|_{L_2(\T^d)}.
\end{equation}
\end{remark}
\medbreak
In Section \ref{s:CNS}, we  used repeatedly 
  the equivalence between  $\|e\|_{L_1},$ $\|\rho-\bar\rho\|^2_{L_2}$ and $\|P(\rho)-P(\bar\rho)\|_{L_2}^2$ when the density is bounded, as stated in the following lemma.
  \begin{lemma}\label{lem:eseq}
      Let $a:=\rho-\bar\rho$ and  $e$ be defined in \eqref{eq:e} with  $P(\rho)=\rho^{\gamma}$ for some $\gamma\geq 1.$  Then, provided
     $0\leq \rho \leq \rho^*,$  
 there exist a positive constant $c_\gamma,$ an increasing  function $F_1$ on $\R_+$ with $F_1(0)=1$
 depending only on $\gamma,$ and an absolute constant $C$  such that 
$$\bar\rho^{\gamma-1}a\leq P(\rho)-P(\bar\rho)\leq  F_1(\rho^*/\bar\rho)\bar\rho^{\gamma-1} a \andf 
c_\gamma\bar\rho^{\gamma-2} a^2\leq e(\rho)\leq C \bar\rho^{\gamma-2} a^2 F_1(\rho^*/\bar\rho).$$
  \end{lemma}
  \begin{proof} By virtue of the mean value formula, we have
  \begin{equation}\label{eq:F1}
  P(\rho)-P(\bar\rho)= \bar\rho^{\gamma-1}(\rho-\bar\rho)F_1(\rho/\bar\rho)\with F_1(s):= \gamma \int_0^1(1+\tau(s-1))^{\gamma-1}\,d\tau.
  \end{equation}
  The function $F_1$ is increasing on $\R_+$ with value $1$ at $0,$ which   gives  the first inequality.
  To prove the second one, from the definition of  $F_1,$ we observe that
      $$\bar\rho^{\gamma-1}\rho\int_{\bar\rho}^\rho\frac{s-\bar\rho}{s^2}\,ds = 
      \rho \int_{\bar\rho}^{\rho} \frac{P(s)-P(\bar\rho)}{s^2F_1(s/\bar\rho)}\,ds,$$
which implies, owing to the definition of $e$ that
      $$\bar\rho^{\gamma-1}\rho \int_{\bar\rho}^{\rho} \frac{s-\bar\rho}{s^2}\,ds\leq e(\rho)=\rho \int_{\bar\rho}^{\rho} \frac{P(s)-P(\bar\rho)}{s^2}\,ds\leq F_1(\rho^*/\bar\rho)\,\rho \int_{\bar\rho}^{\rho} \frac{(s-\bar\rho)}{s^2}\,ds.$$
      Then, in light of the facts that
      $$\rho \int_{\bar\rho}^{\rho} \frac{s-\bar\rho}{s^2}\,ds
      =\rho\log\Bigl(\frac\rho{\bar\rho}\Bigr)+\bar\rho-\rho,$$
      $$
    \underset{\rho\to 0}{\lim} \frac{\rho\log(\rho/\bar\rho)+\bar\rho-\rho}{(\rho-\bar\rho)^2}=\frac1{\bar\rho}  \andf \underset{a\to 0}{\lim} \frac{(\bar\rho+a)\log(1+a/\bar\rho)-a}{a^2}=\frac 1{2\bar\rho},$$
    it is easy to complete the proof.
  \end{proof}

\begin{small}	 

\end{small}

\bigbreak\bigbreak
\noindent\textsc{Univ Paris Est Creteil, Univ Gustave Eiffel, CNRS, \\ LAMA UMR8050, F-94010 Creteil, France}
\par\nopagebreak
E-mail addresses: raphael.danchin@u-pec.fr,
shan.wang@u-pec.fr


\begin{thebibliography}{99}

 \bibitem{BCD}
   H.  Bahouri,  J.-Y.  Chemin and   R.  Danchin:  \emph{Fourier  Analysis  and  Nonlinear  Partial  Differential  Equations,}  Grundlehren  der  Mathematischen  Wissenschaften,  vol. {\bf 343},  Springer-Verlag,  Berlin,  Heidelberg,  2011.

      \bibitem{RPYS} R. Coifman, P.-L. Lions, Y. Meyer and S. Semmes: Compensated compactness and Hardy spaces,\emph{ J. Math. Pures Appl.}, {\bf 72} (1993), no.3, 247-286.
   
     

   \bibitem{BD1997}B. Desjardins: Regularity of weak solutions of the compressible isentropic Navier-Stokes equations,\emph{Comm. Partial Differential Equations}, {\bf 22} (1997), no.5-6, 977-1008.
    \bibitem{CDARMA} F. Charve and R. Danchin:  A global existence result for the compressible Navier-Stokes equations in the critical $L^p$ framework, \emph{Archive for Rational Mechanics and Analysis}, {\bf 198} (2010), no.1,  233-271.
    
    \bibitem{DM-adv}  Danchin, R.  and  Mucha, P.B. 
    Compressible Navier-Stokes system : large solutions and incompressible limit, 
    \textit{Advances in Mathematics}, {\bf 320} (2017), 904--925.
    
  \bibitem{DM1} R. Danchin and P.B. Mucha:  The incompressible Navier-Stokes equations in vacuum, {\em Communications
   on Pure and Applied Mathematics}, {\bf 52} (2019), 1351--1385.
   
   \bibitem{DM-ripped} R. Danchin and P.B. Mucha: Compressible Navier-Stokes equations with ripped density, {\em Communications
   on Pure and Applied Mathematics}, {\bf 76} (2023), 3437--3492.    
   

   \bibitem{DM3}R. Danchin and P.B. Mucha:  From compressible to incompressible inhomogeneous flows in the case of large data. {\em Tunis. J. Math.}, {\bf 1} (2019), 127–149. 
   \bibitem{DMP} R. Danchin, P.B. Mucha and T. Piasecki:  Stability of the density patches problem with vacuum for incompressible inhomogeneous viscous flows, \emph{Ann. IHP, Anal. Non Lin.} (2023).

\bibitem{DT}  R. Danchin and  P. Tolksdorf: 
Critical regularity issues for the compressible Navier-Stokes system in bounded domains, {\bf 387} (2023), 1903--1959.

 \bibitem{Des-CPDE} B. Desjardins: Regularity of weak solutions of the compressible isentropic Navier-Stokes equations, \textit{ Comm. Partial Differential Equations}, {\bf 22}  (1997),  
 no. 5-6, 977--1008. 


\bibitem{Feireisl}  E. Feireisl: \emph{Dynamics of viscous compressible fluids.} Oxford Lecture Series in Mathematics and its Applications, {\bf 26}. Oxford University Press, Oxford, 2004. 
 
    \bibitem{LG}   L. Grafakos: \emph{Classical and Modern Fourier Analysis,} Prentice Hall, 2006.


\bibitem{Hoff}  D. Hoff: Global well-posedness of the Cauchy problem for the Navier-Stokes equations
of nonisentropic flow  with discontinuous initial data, {\it Journal of Differential Equations}, {\bf 95}  (1992),  33--74.    


  \bibitem{Lions} P.-L. Lions: \emph{Mathematical Topics in Fluid Mechanics. Vol. 1. Incompressible Models}, Oxford Lecture Series in
Mathematics and its Applications,  vol. 3, Oxford University Press, 1996. 

\bibitem{MN} A. Matsumura and T. Nishida: 
The initial value problem for the equations of motion of viscous and 
heat-conductive gases, {\it  J. of  Math. of  Kyoto University},
 {\bf 20}  (1980),   67--104.



\bibitem{Nash} J. Nash: Le problème de Cauchy pour les équations différentielles d’un fluide général,
\emph{Bulletin de la Soc. Math. de France}, {\bf  90} (1962),  487--497.
    
  \bibitem{WYZ} G. Wu, L. Yao and Y. Zhang:
  Global Stability and Non-Vanishing Vacuum States of 3D Compressible Navier-Stokes Equations,
 \emph{SIAM Journal on Mathematical Analysis}, {\bf 55}(2), (2023),  pages 882--899.  
  
  \bibitem{ZZ} Z. Zhang and R. Zi: Convergence to equilibrium for the solution of the full compressible Navier-Stokes equations, 
  \emph{Annales de l'IHP, Anal. Non Linéaire}, {\bf 37}(2), (2020),  457--488. 
  
  
    
\end{thebibliography}
\end{document}